\documentclass[12pt]{amsart}
\usepackage{amsmath}
\usepackage{amsthm}
\usepackage{amssymb}
\usepackage{amscd, mathtools}

\usepackage{hyperref}
\usepackage{color}
\usepackage{tikz}
\usetikzlibrary{matrix,arrows}

\newcommand{\Hom}{\operatorname{Hom}}
\newcommand{\NN}{\mathbb{N}}

\DeclareMathOperator{\Ass}{Ass}

\newcommand{\hght}{\operatorname{ht}}

\newcommand{\eh}{\operatorname{e}}
\newcommand{\ehk}{\eh_{HK}}

\newcommand{\mf}{\mathfrak}
\newcommand{\frq}[1]{{#1}^{[p^e]}}
\newcommand{\frpe}[1]{{#1}^{[p^e]}}

\DeclareMathOperator{\length}{lg}

\DeclareMathOperator{\rank}{rank}
\DeclareMathOperator{\coker}{coker}

\DeclareMathOperator{\fsig}{s}
\DeclareMathOperator{\asp}{\ell}

\DeclareMathOperator{\Minh}{Minh}
\DeclareMathOperator{\Ann}{Ann}
\DeclareMathOperator{\Supp}{Supp}
\DeclareMathOperator{\Spec}{Spec}

\DeclareMathOperator{\inv}{\epsilon}

\newcommand{\red}[1]{#1_{\operatorname {red}}}

\newtheorem{theorem}{Theorem}
\numberwithin{theorem}{section}
\newtheorem{lemma}[theorem]{Lemma}
\newtheorem{proposition}[theorem]{Proposition}
\newtheorem{corollary}[theorem]{Corollary}

\newtheorem*{statement*}{Statement}
\newtheorem*{theorem*}{Theorem}
\newtheorem*{lemma*}{Lemma}
\newtheorem*{fact*}{Fact}

\theoremstyle{definition}
\newtheorem{definition}[theorem]{Definition}
\newtheorem*{definition*}{Definition}

\newtheorem*{example*}{Example}

\newtheorem{remark}[theorem]{Remark}
\newtheorem{claim}{Claim}

\begin{document}

\title{An invitation to equimultiplicity of F-invariants}

\author{Ilya Smirnov}
\address{BCAM -- Basque Center for Applied Mathematics, Bilbao, Spain \quad and \quad IKERBASQUE, Basque Foundation for Science, Bilbao, Spain}
\email{ismirnov@bcamath.org}

\maketitle

\begin{center}
{\textit{Dedicated to Craig and Mel  for putting the wind in our sails. }}
\end{center}

\section{Introduction}

This note grew from the lectures that I delivered at the ICTP school, but the content has changed a lot. Since there are several good introductions to F-invariants (such as \cite{MaPolstraBook, CaminataDe, HunekeSurvey}), I decided to omit the proofs of the basic results and to focus instead on the notion of \emph{equimultiplicity} and its role in the theory of Hilbert--Kunz multiplicity and F-signature. My motivation was the lack of  a good introduction to equimultiplicity in general and the feeling that there is still a lot to be done in the study of equimultiplicity for \emph{F-invariants}, i.e., invariants derived from the Frobenius. 

\emph{Equimultiplicity} is defined for a numerical invariant of local ring, let $\inv$ be one of these. By localizing at primes, $\inv$ defines a function on the spectrum. A prime $\mf p \in \Spec R$ is $\inv$-\emph{equimultiple} if $\inv(\mf p) = \inv(\mf q)$ for all $\mf p \subseteq \mf q$; this condition appears naturally from the stratification of the spectrum by the values of $\inv$.  In order to work with $\inv$-equimultiplicity it is essential to find convenient necessary and sufficient conditions, perhaps a criterion, for $\mf p$ to be $\inv$-equimultiple. In Section~\ref{sec: equimultiplicity intro}, I will discuss in details an equimultiplicity approach to the following three problems:
\begin{enumerate}
\item showing that an extremal value (which is $1$ for the invariants considered here) of $\inv$ determines whether a ring is regular,
\item proving that $\inv$ is strongly upper semicontinuous,
\item resolving singularities. 
\end{enumerate}

This note pursues the first two applications for Hilbert--Kunz multiplicity and F-signature. It should be stressed that equimultiplicity theories for multiplicity and the Hilbert--Samuel function were crucial in the resolution of singularities in characteristic $0$. It could happen that one day F-invariants will serve a similar role in the resolution of singularities in positive characteristic -- this is a good enough reason to study their equimultiplicity. With this point of view, the applications in this note are a useful byproduct and a test of tools.

\subsection*{The structure of this note}
In Section~\ref{sec: equimultiplicity intro}, I will start with a brief overview of  equimultiplicity and its main uses. Section~\ref{sec: F-intro} introduces F-signature and Hilbert--Kunz multiplicity and can be safely skipped by experienced readers. 
Section~\ref{sec: F-signature} considers equimultiplicity for F-signature, it is
largely taken from my work with Polstra \cite{PolstraSmirnov}.

The following two sections are more novel. In Section~\ref{convergence} I will prove a uniform convergence estimate for Hilbert--Kunz multiplicity and present its applications; the treatment here is a bit different and inspired by Huneke's survey \cite{HunekeSurvey}. Building on the tools obtained through uniform convergence,  Section~\ref{sec: Hilbert-Kunz} develops a theory of Hilbert--Kunz multiplicity for primes of dimension $1$. An important difference with \cite{SmirnovEqui} is that it is not assumed that $R/\mf p$ is regular. 
This leads to two novel applications: a rigidity theorem for Hilbert--Kunz multiplicity in weakly F-regular rings (\cite{PolstraSmirnov} dealt with strongly F-regular rings) and a different proof of the Watanabe--Yoshida theorem. 
This proof is a modification of the proof given by Huneke and Yao in \cite{HunekeYao} -- I replaced their use of the symbolic power $\mf p^{(n)}$ by $(\frq{\mf p})^*$, this ideal is $\mf p$-primary due to equimultiplicity. This change makes my proof more technical, but it also makes the proof more \emph{conceptual} as it fits in the framework that can be applied to other invariants. The note is finished with suggestions for further research.

\subsection*{Acknowledgments}
I thank Linquan Ma for many fruitful conversations, Thomas Polstra for collaborating on \cite{PolstraSmirnov}, and Craig Huneke, Austyn Simpson, and the anonymous referee for helping me improve this manuscript.
I am indebted to Craig Huneke and Mel Hochster for their mathematics and mentorship through the years. Like most papers in ``characteristic $p$ methods'', this note has something from Mel and Craig on almost every page. I want to highlight that the presented necessary condition for Hilbert--Kunz equimultiplicity is given in terms of \emph{tight closure}.

This note grew from my thesis, the part developed while Craig was still in Kansas, and it is influenced by every weekly meeting we had back then. In particular, I still remember clearly that during one meeting I showed him that equimultiplicity implies that the saturation of Frobenius powers was contained in their tight closures and Craig suggested that this should imply that the tight closure itself should be saturated, see Theorem~\ref{thm: equimultiple property}. At that time, Craig also introduced me to Lipman's article on equimultiplicity \cite{Lipman} which had a lot of influence on me through the years. 


\section{Equimultiplicity}\label{sec: equimultiplicity intro}

The purpose of this section is to briefly discuss the notion of equimultiplicity and its potential applications. 

First, in the following a local ring $(R, \mf m)$ is defined 
as a Noetherian unital ring such that the set of all non-invertible elements forms an ideal $\mf m$. Let $\inv$ be an invariant of local rings, it defines a function on the spectrum of a Noetherian ring $R$ by setting $\inv (\mf p) = \inv (R_\mf p)$. 

If $\inv$ takes values in a totally ordered set $\Lambda$, it is desirable that $\inv\colon \Spec R \to \Lambda$ is a map of posets. 
This property comes from the intuitive expectation that if $\mf p \subseteq \mf q \in \Spec R$, then the singularity of $R_\mf p$ cannot be worse than the singularity of $R_\mf q$. For example, the Hilbert--Samuel multiplicity, $\eh(R)$, is a positive integer, so the 
inequality becomes $\eh(\mf p) \leq \eh(\mf q)$.
For some other invariants, such as F-signature, the inequality should be opposite. 

A prime $\mf p$ is $\inv$-equimultiple if $\inv(\mf p) = \inv (\mf q)$ for all $\mf p \subseteq \mf q$; the intuitive meaning is that  the singularity, as measured by $\inv$, does not change along $\mf p$. In order to work with this condition, it is necessary to characterize it in a more intrinsic way. 
As an example, the following theorem presents the classical case of Hilbert--Samuel multiplicity.

\begin{theorem}[{\cite[Theorem~2]{Dade}, see also \cite[Corollary, page 121]{Lipman}}]\label{thm: multiplicity criterion}
Let $(R, \mf m)$ be a formally equidimensional local ring and $\mf p$ be a prime ideal such that $R/\mf p$ is regular. 
Then $\eh (\mf m) = \eh (\mf p)$ if and only if 
the analytic spread of $\mf p$, $\asp(\mf p)$, is minimal, i.e., the dimension of the fiber cone $\oplus \mf p^n/\mf m\mf p^n$ is equal to $\hght \mf p$.
\end{theorem}

\begin{remark}
In literature, any ideal $I$ such that $\asp(I) = \hght (I)$ is called equimultiple. 
While this condition has a connection with multiplicities (see \cite[Theorem~4]{Lipman}), it is not equivalent to $\eh$-equimultiplicity if $R/\mf p$ is not regular.
If $R/\mf p$ is not regular, I do not know of any necessary condition for Hilbert--Samuel equimultiplicity. 

For example, any prime ideal in a regular ring is $\eh$-equimultiple by the definition, but the height and the analytic spread need not be equal.
A family of examples is given by defining ideals of a monomial curve with $3$ generators, such as the defining prime $\mf p = (x^3 - yz, y^2- xz, z^5  - x y^3) \subset k[x,y, z]$ of the monomial curve $k[t^3, t^4, t^5]$. 
As discussed on page 258 of \cite{HunekeDseq} the defining primes are generated by a $d$-sequence, so the fiber cone is isomorphic to a polynomial ring in $3$ variables by \cite[Theorem~2.2]{HunekeDseq}. 
\end{remark}

While the two invariants of this note, Hilbert--Kunz multiplicity and F-signature, 
take values in real numbers, it should be noted that a theory of equimultiplicity can be developed for invariants with values in a more general poset. Two classically studied examples are the Hilbert function, which takes values in sequences of positive integers with an appropriate order, see \cite{Bennett}, and the Hilbert polynomial.

\subsection{Connection with semicontinuity}

\begin{definition}
Let $\inv \colon X \to \mathbb{R}$ be a function on a topological space $X$. 
We say that $\inv$ is \emph{upper semicontinuous} if the subset $X_{< a} \coloneqq \{x \in X \mid \inv(x) < a\}$
is open for all $a \in \mathbb{R}$. We say that $\inv$ is \emph{strongly upper semicontinuous} if any set of the form $X_{\leq a} \coloneqq \{x \in X \mid \inv(x) \leq a\}$ is open. Since $X_{< a} = \cup_{b < a} X_{\leq b}$, strong upper semicontinuity is indeed a stronger notion. 

Lower semicontinuity and strong lower semicontinuity are defined similarly, by using $X_{> a}$ and $X_{\geq a}$. 
\end{definition}

Semicontinuity is a notion of continuity 
appropriate for algebraic geometry and it has a number of important consequences in the theory of F-invariants, see for example \cite{EnescuShimomoto, SPYGlobal, SmirnovTucker} and the discussion in \ref{sub: blowup}. 
In addition, strong semicontinuity provides a constructible stratification, \emph{i.e.,} 
the sets $X_{= a} \coloneqq X_{\leq a} \cap X_{\geq a}$ are locally closed. 

In practice, it might be easier to establish strong semicontinuity, for example, due to the following connection with equimultiplicity that comes from 
Nagata's criterion for openness. The criterion asserts that a set $U \subseteq \Spec R$ 
is open if and only if for all pairs $\mf p \subseteq \mf q \in \Spec R$ 
the following conditions hold 
\begin{enumerate}
\item if $\mf q \in U$ then $\mf p \in U$, and
\item if $\mf p \in U$ there exists an element $s \notin \mf p$
such that whenever $s \notin \mf q$ then $\mf q \in U$. 
\end{enumerate}
By applying this to $X_{\leq a}$, the following characterization follows immediately.

\begin{proposition}\label{prop: Nagata}
The function $\inv \colon \Spec R \to \mathbb {R}$ is strongly upper semicontinuous 
if and only if $\inv$ satisfies the localization inequality, \emph{i.e.,} $\inv (R_\mf p) \leq \inv(R_\mf q)$ when $\mf p \subseteq \mf q$, and for any prime $\mf p$ there is $s\notin \mf p$ such that $\mf pR_s$ is $\inv$-equimultiple. 
\end{proposition}

With a convenient sufficient condition for $\inv$-equimultiplicity, such as Theorem~\ref{thm: multiplicity criterion}, 
Proposition~\ref{prop: Nagata} can be easily used to show semicontinuity. 

\begin{corollary}
Multiplicity is upper semicontinuous on the spectrum of an excellent ring. 
\end{corollary}
\begin{proof}
It was shown in \cite{Dade} that multiplicity localizes if $R$ is excellent\footnote{By a surprising result of \cite{LL}, the localization property for all rings is equivalent to the notorious conjecture of Lech on multiplicity in faithfully flat extensions.}. Furthermore, for any prime $\mf p$ 
the quotient $R/\mf p$ is still excellent, so its regular locus is open by the definition.
Thus, there is an element $u \notin \mf p$
such that $(R/\mf p)_u$ is regular. Second, the analytic spread of $\mf pR_\mf p$ is $\dim R_\mf p = \hght \mf p$, so it is possible to find $v \notin \mf p$ 
so that the same equality holds in $R_v$.
It remains to apply Theorem~\ref{thm: multiplicity criterion} in $R_{uv}$. 
\end{proof}

On the other hand, Corollary~\ref{not locally constant} will use 
Proposition~\ref{prop: Nagata} to show that Hilbert--Kunz multiplicity is \emph{not} strongly upper semicontinuous.

\subsection{Connection to resolution of singularities} \label{sub: blowup}
The notion of equimultiplicity is also closely connected to the main approach to the resolution of singularities, and essentially originates, in the case of Hilbert function,  from Hironaka's work \cite{Hironaka}. I will briefly outline this connection, and, for simplicity, I will assume that $\inv(R)$ is upper semicontinuous (a lower semicontinuous function is completely analogous). 

Since $X = \Spec R$ is quasi-compact, the open cover $X = \cup X_{< a}$ is finite, hence $\inv$ must attain its supremum on $X$. Let $M$ to denote the maximal value of $\inv$ on $X$. By semicontinuity, the maximal locus of $\inv$, i.e., the set $X_{\max} \coloneqq \{\mf p \mid \inv (\mf p) = M\} = X \setminus X_{< M}$, is closed. 
As measured by $\inv$,  $X_{\max}$ is the locus of the worst singularity, thus
a natural attempt to improve the singularity is to blow up $X_{\max}$ or its component. 
The key difficulty is to show that singularity improves after blowing up. 

Because any prime ideal in $X_{\max}$ is $\inv$-equimultiple, 
$\inv$-equimultiplicity should help to understand the effect of the blowing up. The theories of equimultiplicity for Hilbert--Samuel multiplicity and Hilbert functions were developed for this reason. As a guideline consider the result of Dade that controls the multiplicity after blowing up an $\eh$-equimultiple ideal.

\begin{theorem}[{\cite[Theorem~3]{Dade}, \cite[Corollary~31.2]{EquimultiplicityBook}}]\label{thm: Dade blowup}
Let $(R, \mf m)$ be a formally equidimensional local ring. 
If $\mf p$ is a prime ideal such that $\asp (\mf p) = \hght \mf p$ 
then any local ring $S$ of the blowing up\footnote{In other words, $S$ is a homogeneous localization of the Rees algebra at a relevant prime that contracts to $\mf m$.} of $\mf p$ that dominates $R$  satisfies the inequality
\[
\eh(S) \leq \eh (R/\mf p) \eh (R_\mf p).
\]
\end{theorem}

This theorem only says that the singularity does not get worse, but 
multiplicity is a very crude measure. In fact, in characteristic $0$ one can build an invariant $\inv$ so that a resolution of singularities is obtained by consecutively blowing up the maximal locus of $\inv$. Loosely speaking, the constructions of $\inv$ originating from Hironaka's work adjoin to the Hilbert--Samuel function additional invariants until the resulting sequence of invariants $\inv$ decreases with each blow-up. I refer to \cite{EquimultiplicityBook} for more details about equimultiplicity in the Hilbert--Samuel theory.

\subsection{Detection of singularities}
Equimultiplicity also provides an approach to showing that a given invariant $\inv$ 
detects singularities. For simplicity, assume $\inv (R) \in [1, \infty)$, so that the goal is to show that if  $(R, \mf m)$ a local ring (perhaps with extra assumptions) and $\inv(\mf m) = 1$, then $R$ is regular. 
If $\inv$ localizes well, then the condition $\inv(R) = 1$ implies that every prime ideal in $R$ is $\inv$-equimultiple. 
This opens an approach by induction on the dimension of $R$: we may take a suitable $\mf p \neq \mf m$, then 
$R_\mf p$ is regular by induction. Thus the goal becomes to show that if $\mf p$ is $\inv$-equimultiple and $R_\mf p$ is regular, then $R$ must be regular.

In this note, this strategy is applied to Hilbert--Kunz multiplicity and F-signature. 
Interestingly, as Theorem~\ref{thm: multiplicity criterion} is limited to the situation where $R/\mf p$ is regular, one cannot use the same strategy to recover Nagata's original criterion of $\eh(R) = 1$.  

\section{An introduction to F-invariants}\label{sec: F-intro}
The purpose of this section is to discuss the basics of the positive characteristic 
commutative algebra and introduce Hilbert--Kunz multiplicity and F-signature. 
The proofs and more comprehensive treatments can be found in \cite{HunekeSurvey, MaPolstraBook, CaminataDe}.

A key result for the study of singularities in positive characteristic is the characterization of regular rings due to Kunz. Essentially all proofs showing 
that a certain invariant detects singularities utilize this famous theorem. 
In the theorem, $F_*^e R$ denotes an $R$-module obtained as the target 
of the $e$-th iterate of the Frobenius morphism $F^e \colon R \to R$.

\begin{theorem}[\cite{Kunz1}]\label{theorem Kunz}
Let $(R, \mf m)$ be a local ring of prime characteristic $p > 0$.
Then $R$ is regular if and only if $F^e_* R$ is a flat $R$-module for some (equivalently, all) positive integer $e$.
\end{theorem}

If, in addition, $R$ is \emph{F-finite} (that is, $F_*^e R$ is finitely generated), then 
flatness is upgraded to freeness and the ``failure of freeness'' of $F_*^e R$
can be measured numerically to provide criteria for regularity. 
As an example, assume that $R$ is a domain with the field of fractions $L$ and define the generic rank of $M$ as its rank at the generic point, $\dim_{L} M \otimes_R L$. Consider two functions 
\[
f(M) = \frac{\min \{N \mid \text{ there is } R^{\oplus N} \to M \to 0\}} {\text{generic rank of } M} =
\frac{\text{minimal number of generators of } M} {\text{generic rank of } M} 
\]
and 
\[
g(M) = \frac{\max \{N \mid \text{ there is } M \to R^{\oplus N} \to 0\}}{\text{generic rank of } M} = 
\frac{\max \{N \mid M = \oplus^N R \oplus M'\}}{\text{generic rank of } M}.
\] 
It is easy to see that $f(M) \in [1, \infty)$, $g(M) \in [0, 1]$, and 
$f(M) = 1$ if and only if $M$ is free if and only if $g(M) = 1$. 

Kunz's theorem gives numerical tests for singularity through the application of either function to $F_*^e R$ for any $e$. But it happens that the values change with $e$ in a controllable way leading to the following definitions. 

\begin{definition}[Kunz \cite{Kunz2}, Monsky \cite{Monsky}]
Let $(R, \mf m)$ be an F-finite local domain of characteristic $p > 0$. 
Then the Hilbert--Kunz multiplicity of $R$ is 
\[
\ehk (R) = \lim_{e \to \infty} \frac{\text{minimal number of generators of } F_*^e R}{\text{generic rank of } F_*^e R}.
\]
\end{definition}

\begin{definition}[Smith--Van der Bergh \cite{SmithVdB}, Huneke--Leuschke \cite{HunekeLeuschke}]\label{def: Fsig}
Let $(R, \mf m)$ be an F-finite local domain of characteristic $p > 0$. 
Then the F-signature of $R$ is 
\[
\fsig (R) = \lim_{e \to \infty} \frac{\max \{N \mid F_*^e R = \oplus^N R \oplus M\}}{\text{generic rank of } F_*^e R}.
\]
\end{definition}

It requires work to show that these limits exist and also detect singularity. This will be discussed in Section~\ref{convergence}. 

\subsection{The generic rank theorem}
In order to work with these invariants it is convenient to restate them in a different way. 
The starting place is another fundamental result of Kunz \cite{Kunz2} that computes 
the generic rank. 

\begin{theorem}[Kunz]\label{thm: Kunz rank}
Let $(R, \mf m, k)$ be an F-finite local domain of dimension $d$.
Then for any $e > 0$ the generic rank of $F_*^e R$ is 
$p^{ed} [k^{1/p^e} : k]$. 
\end{theorem}
\begin{proof}
The idea of the proof is that it reduces, via completion and Cohen's structure theorem, 
to the easiest case where $R = k[[x_1, \ldots, x_d]]$. In this case the computation is very explicit, because $F_*^e R \cong R^{1/p^e} = k^{1/p^e}[[x_1^{1/p^e}, \ldots, x_d^{1/p^e}]]$ as $R$-algebras by mapping $F_*^e r \mapsto r^{1/p^e}$. 
Let $B = k^{1/p^e}[[x_1, \ldots, x_d]]$. Since $k^{1/p^e}$ is free over $k$, $B$ is a free $R$-module of rank $[k^{1/p^e} : k]$.
Furthermore, $k^{1/p^e}[[x_1^{1/p^e}, \ldots, x_d^{1/p^e}]]$ 
is a free $B$-module of rank $p^{ed}$ by giving explicitly a basis 
$x_1^{a_1/p^e}\cdots x_d^{a_d/p^e}$ where $0 \leq a_i \leq p^e - 1$.
\end{proof}

\begin{corollary}\label{cor: Kunz primes}
Let $R$ be an F-finite ring and $\mf p \subsetneq \mf q$ be prime ideals.
Then for any $e > 1$  
\[
[k(\mf p)^{1/p^e} : k(\mf p)] =  p^{e\dim R_\mf q/\mf pR_\mf q} [k(\mf q)^{1/p^e} : k(\mf q)].
\] 
\end{corollary}
\begin{proof}
Follows from the theorem by passing to $R_\mf q/\mf pR_\mf q$.
\end{proof}

\subsection{Hilbert--Kunz multiplicity via colengths.}
In order to employ ideal-theoretic techniques, it is convenient to restate the first definition of Hilbert--Kunz multiplicity. 

In the rest of the paper $\length_R (M)$ (or $\length (M)$ if there is no risk of ambiguity) will be used the length of an $R$-module $M$.

\begin{lemma}\label{lem: numgen}
Let $(R, \mf m)$ be an F-finite local ring with the residue field $k$. Then for any integer $e \geq 0$
the minimal number of generators of $F_*^e R$ is 
$[k^{1/p^e} : k]\length (R/\frpe{\mf m})$, 
where 
$\frpe{\mf m}$ is the ideal generated by $\{x^{p^e} \mid x \in \mf m\}$.
\end{lemma}
\begin{proof}
Since $R$ is local the minimal number of generators of any module is 
$\length (M/\mf mM)$. Then $\mf m \cdot F_*^e R$ consists of elements of the form
$F_* \left (\sum_i r_i m_i^{p^e} \right )$, hence 
\[
\length (F_*^e R/\mf m F_*^e R) = \length (F_*^e R/\frpe{\mf m}).
\]
By the definition of length, for $L = \length (R/\frpe{\mf m})$ there is a filtration 
\[
\frpe{\mf m} = M_0 \subsetneq \cdots \subsetneq M_L = R
\]
with $M_{i + 1}/M_i \cong k$. 
Since $F_*^e \bullet$ is exact, it follows that 
$
\length (F_*^e R/\mf m F_*^e R) = [F_*^e k: k ] L.
$
Since $k$ is a field, $F_*^e k \cong k^{1/p^e}$ as $k$-algebras. 
\end{proof}

It follows from Theorem~\ref{thm: Kunz rank} and the definition that 
\[
\frac{\text{number of generators of } F_*^e R}{\text{generic rank of } F_*^e R}
=\frac{[k^{1/p^e}:k] \length (R/\frq{\mf m})}{p^{e\dim R} [k^{1/p^e} : k]}
= \frac{\length (R/\frq{\mf m})}{p^{e\dim R}}.
\]
This immediately extends the definition of Hilbert--Kunz multiplicity to all local rings of positive characteristic and all $\mf m$-primary ideals. 

\begin{definition}\label{def: general HK}
If $(R, \mf m)$ is a local ring of characteristic $p > 0$, 
$I$ an $\mf m$-primary ideal, and $M$ a finitely generated $R$-module, then 
the Hilbert--Kunz multiplicity of $I$ on $M$ is 
\[
\ehk (I; M) =  \lim_{e \to \infty} \frac{\length (M/\frpe{I}M)}{p^{e\dim R}}.
\]
\end{definition}

It is a custom to denote $\ehk (I) \coloneqq \ehk (I; R)$, moreover, in line with the previous definition of Hilbert--Kunz multiplicity and to highlight its meaning as a measure of singularity of $R$, it is common to denote $\ehk (R) \coloneqq \ehk (\mf m)$. Unfortunately, this may cause a certain ambiguity, but I will not use $I = R$ in 
Definition~\ref{def: general HK} as it is not $\mf m$-primary.  

It is now time to record several needed properties of Hilbert--Kunz multiplicity. 

\begin{proposition}[Some properties of Hilbert--Kunz multiplicity]\label{prop: HK properties}
Let $I$ be an $\mf m$-primary ideal. Then the following properties hold.
\begin{enumerate}
\item $\ehk (I; M) = 0$  if $\dim M < \dim R$.
\item $\ehk (I^{[p]}; M) = p^d \ehk (I; M)$.
\item (Additive) If $0 \to L \to M \to N \to 0$ is an exact sequence of finitely generated $R$-modules, then 
$\ehk (I; M) = \ehk (I; L) + \ehk (I; N)$.\label{item: exact sequence}
\item (Associativity) $\ehk (I; M) = \sum_\mf p \ehk (I; R/\mf p) \length (M_\mf p)$,
where the sum ranges over primes such that $\dim R/\mf p = \dim R$.
\label{item: associative}
\item (Lech-type inequality) $\ehk (I) \leq \length (R/I) \ehk (\mf m)$.\label{item: filtration}
\item Similarly, $\ehk (I) \leq \length (J/I) \ehk (\mf m) + \ehk (J)$ for any ideal $I \subseteq J \subset R$.\label{item: filtration 2}
\end{enumerate}
\end{proposition}
\begin{proof}
Most of there properties were observed in \cite{WatanabeYoshida}. I want to comment on the last three. 

First, recall that $M$ has a prime filtration: 
$0 = M_0 \subsetneq M_1 \subsetneq \cdots \subsetneq M_L = M$ 
such that $M_{i+1}/M_i \cong R/\mf p_i$. Thus $\ehk (I; M) = \sum_{i} \ehk (I; R/\mf p_i)$ by (\ref{item: exact sequence}). However, by the first property, 
$\ehk (I; R/\mf p_i) = 0$ unless $\dim R/\mf p_i = \dim R$. 
Since any prime $\mf p$ such that $\dim R/\mf p = d$ is minimal, in the filtration $(M_i)_\mf p$ only remain the quotients such that $R/\mf p_i = R/\mf p$. Thus the localized filtration is a composition series of $M_\mf p$, its length must be $\length (M_\mf p)$.  

The proof of the last two properties is similar. Take a composition series of $J/I$, 
where $J = R$ in (\ref{item: filtration}): 
a sequence $I = I_0 \subsetneq I_1 \subsetneq \cdots \subsetneq I_\ell = J$
where $I_{k + 1}/I_k \cong R/\mf m$, i.e., $I_{k + 1} = (I_k, u_{k})$
where $\mf m u_{k} \subseteq I_k$. 
It follows that $\frpe{\mf m} u_{k}^{p^e} \subseteq \frpe{I_{k}}$, so 
\[
\length (R/\frpe{I}) = \sum_{k = 0}^{\ell - 1} \length (\frpe{I_{k+1}}/\frpe{I_{k}}) + \length (R/\frpe{J}) \leq \sum_{k = 0}^{\ell - 1} \length (R/\frpe{\mf m}) + \length (R/\frpe{J}). 
\]
As $\length (R/\frpe{R})=0$, both (\ref{item: filtration}) and (\ref{item: filtration 2}) follow after dividing by $p^{e\dim R}$ and taking limits. 

\end{proof}

\section{Equimultiplicity for F-signature}\label{sec: F-signature}
I will now present the results on F-signature that were obtained in \cite{PolstraSmirnov}. As a first step, we should verify that the invariant satisfies the localization property. In order to shorten Definition~\ref{def: Fsig}, 
denote $a_e (R) = \max \{N \mid F_*^e R = \oplus^N R \oplus M\}$.

\begin{lemma}
If $(R, \mf m)$ is an F-finite local domain then F-signature satisfies the localization property: if $\mf p$ is a prime ideal then $\fsig(R_\mf p) \geq \fsig(R)$. 
\end{lemma}
\begin{proof}
If $M$ is an arbitrary finitely generated $R$-module then 
the free rank of $M_\mf p$ may only increase: 
if $M = \oplus^n R \oplus M'$ then $M'_\mf p$ may gain free summands.  
Hence, after confirming that $(F_*^e R)_\mf p \cong F_*^e R_\mf p$ 
one immediately sees that
\[
\fsig (R) = \lim_{e \to \infty} \frac{a_e(R) }{\text{generic rank of } F_*^e R}
\leq  \lim_{e \to \infty} \frac{a_e(R_\mf p)}{\text{generic rank of } F_*^e R} = 
\fsig(R_\mf p).
\]
\end{proof}

The key insight of the section is the following lemma that produces an adequate number of direct summands. 

\begin{lemma}\label{lemma: summands}
Suppose  $(R,\mathfrak{m})$ is an $F$-finite domain and $\fsig(R) > 0$. Suppose  $M$ is a finitely generated $R$-module and $e_0$ is a positive integer such that there is a decomposition
\[
F_*^{e_0} R \cong M \oplus N_{0}.
\]
If $M$ has no free direct summand, then for any $e \geq e_0$ there is a decomposition
$
F_*^{e} R \cong R^{\oplus a_e(R)}\oplus M^{\oplus b_e} \oplus N_{e}
$
such that 
\[
\lim_{e\to \infty}\frac{b_e}{\rank(F^e_*R)} = \frac{\fsig(R)}{\rank(F^{e_0}_*R)} > 0.
\]
\end{lemma}
\begin{proof}
For simplicity, I will write $a_e$ instead of $a_e(R)$. 
For any $e\in \mathbb{N}$ write $F^e_*R\cong R^{\oplus a_e}\oplus M_e$. Then 
\[
F^{e+e_0}_*R\cong F^{e_0}_*R^{\oplus a_e}\oplus F^{e_0}_*M_e
\cong \oplus^{a_e} \left (M \oplus N_0 \right) \oplus F^{e_0}_*M_e
\cong \oplus^{a_e}  M \oplus \left(  \oplus^{a_e} N_{0} \oplus F^{e_0}_*M_e\right).
\]
Therefore via this decomposition 
 \[
 \lim_{e\to \infty}\frac{b_e}{\rank(F^e_*R)} = \lim_{e\to \infty}\frac{a_{e-e_0}}{\rank(F^e_*R)}=\frac{\fsig(R)}{\rank(F^{e_0}_*R)}.
 \]
 
It remains to observe that the maximal rank of a free summand of any finitely generated $R$-module is an invariant independent of the choice of a direct sum decomposition.  This is because $R$ is a direct summand of $N$ if and only if $\hat{R}$ is a direct summand of $\hat{N}$ and a complete local ring satisfies the Krull--Schmidt condition. Therefore, as $M$ has no free summand, 
it is possible to decompose 
$\oplus^{a_e} N_{0} \oplus F^{e_0}_*M_e \cong R^{\oplus a_{e + e_0}} \oplus N_{e + e_0}$.
\end{proof}

The proof of the localization property leads quite naturally
to the following observation on $\fsig$-equimultiplicity.
While a truly intrinsic characterization of $\fsig$-equimultiplicity is not known, the result still suffices for showing that F-signature detects singularity. 

\begin{theorem}[Rigidity property]\label{theorem: F-signature is rigid}
Let $(R,\mathfrak{m})$ be an $F$-finite local domain  
and $\mf p\in \Spec R$ be a prime ideal. 
Suppose that $\fsig(R) > 0$. Then $\fsig(R)=\fsig(R_\mf p)$ if and only if $a_e(R)=a_e(R_\mf p)$ for every $e\in \mathbb{N}$.
\end{theorem}
\begin{proof}
If $a_e(R)=a_e(R_\mf p)$ for every $e\in \NN$ then clearly $\fsig(R)=\fsig(R_\mf p)$. 
Otherwise, suppose that $a_{e_0}(R_\mf p)>a_{e_0}(R)$ for some $e_0$ and write $F^{e_0}_*R\cong R^{\oplus a_{e_0}}\oplus M_{e_0}$. Then $(M_{e_0})_\mf p$ has a free $R_\mf p$-summand. 
By Lemma~\ref{lemma: summands} 
for each $e\in \mathbb{N}$ we may decompose
 \[
 F^e_*R\cong R^{\oplus a_e(R)}\oplus M_{e_0}^{\oplus b_e}\oplus N_e
 \]
 where $b_e$ grows as fast as $\rank (F_*^e R)$. 
 Localizing at the prime $P$ we see that $a_e(R_\mf p)\geq a_e(R)+b_e$ and
 \[
 \fsig(R_\mf p) = \lim_{e\to \infty}\frac{a_e(R_\mf p)}{\rank(F^e_*R)}
 \geq \lim_{e\to \infty}\frac{a_e(R) + b_e}{\rank(F^e_*R)}
 = \fsig(R) \frac{1 + \rank(F^{e_0}_*R)}{\rank(F^{e_0}_*R)} > \fsig(R).
 \]
\end{proof}

An application useful for experiments is that if there is a single $e$ such that 
$a_e(R) \neq a_e(R_\mf p)$ then $\fsig(R) \neq \fsig(R_\mf p)$. 

\begin{remark}
In the case where $\fsig(R) = 0$ the F-signature functions can be different.  
As an example, take $R = \mathbb{F}_p [[x,y, z]]/(x^2 - y^2z)$ with $p > 2$.
It was computed in \cite[4.3.2]{BlickleSchwedeTucker}
that $a_e(R) \approx p^e/2$. On the other hand, 
the prime ideal $\mf p = (x, y)$ is the F-splitting prime of $R$ in the sense 
of \cite{AberbachEnescu}.
Then by \cite[Proposition~3.6, Corollary~3.4]{AberbachEnescu} 
$a_e(R_\mf p) = p^e$. 
\end{remark}

\subsection{F-signature and non-singularity}
In contrast with the multiplicity-like invariants, F-signature has two distinguished values, $0$ and $1$. 
It was shown by Aberbach--Leuschke in \cite{AberbachLeuschke} (see also \cite{PolstraTucker} for alternative proofs) that F-signature is positive if and only if $R$ is \emph{strongly F-regular}. It will be used in the proof that 
a strongly F-regular local ring is a domain, \cite[Lemma~3.2]{MaPolstraBook}.

\begin{theorem}[{\cite[Corollary~16]{HunekeLeuschke}}]\label{theorem signature 1 iff regular} Let $(R,\mf m)$ be an $F$-finite local ring of prime characteristic $p$. Then $\fsig(R)=1$ if and only if $R$ is a regular local ring.
\end{theorem}

\begin{proof}
Since $\fsig(R) > 0$, $R$ is a domain, so $R_{(0)}$ is a field and 
$a_e(R_{(0)})=\rank(F^e_*R)$. By the localization inequality, 
$\fsig(R_{(0)}) = \fsig(R)$, so 
$a_e(R)=a_e(R_{(0)})=\rank(F^e_*R)$ by Theorem~\ref{theorem: F-signature is rigid}. Therefore $F^e_*R$ is a free $R$-module and $R$ is a regular local ring by Theorem~\ref{theorem Kunz}.
\end{proof}

\section{Uniform convergence of F-invariants}\label{convergence}
The theory of $\ehk$-equimultiplicity requires a result interchanging the order of  limits. 
Following \cite{SmirnovEqui} the formula will be derived from a uniform convergence estimate. This type of convergence estimates first 
appeared in \cite{Tucker}. The proof presented here borrows from Huneke's survey \cite{HunekeSurvey}, these arguments are slightly different from 
\cite{Tucker, SmirnovEqui}.

\subsection{Convergence estimates}

One way to think through the proof of existence of Hilbert--Kunz multiplicity is through an equivalence relation on modules, where $M \sim N$ if there exists a module $T$ such that $\dim T < \dim R$ and $|\length (M/IN) - \length (N/IN)| < \length (T/IT)$ for all $\mf m$-primary ideals $I$. 
The uniform convergence argument imposes a further control on the difference module $T$. 

\begin{lemma}\label{lem: isomorphic equivalence}
Let $(R, \mf m)$ be a local ring and $M, N$ be finitely generated $R$-modules. 
Suppose that  $M_\mf p \cong N_\mf p$ for every prime $\mf p$ with $\dim R/\mf p = \dim R$.
Then there exist exact sequences 
\[
M \to N \to T_1 \to 0 \text{ and } N \to M \to T_2 \to 0
\]
such that $\dim T_1, \dim T_2 < \dim R$. 
In particular, if $T = T_1 \oplus T_2$, then for any $\mf m$-primary ideal $I$
\[
|\length (M/IN) - \length (N/IN)| < \length (T/IT).
\]
\end{lemma}
\begin{proof}
There are only finitely many primes $\mf p_i$ with $\dim R/\mf p_i = \dim R$. 
Set $S = R \setminus \cup \mf p_i$. By the assumption 
$S^{-1} M \cong S^{-1} N$.
Since $\Hom$ commutes with localization for finitely presented modules, 
there are maps $\psi\colon M \to N$ and $\phi \colon N \to M$ 
which become isomorphisms after localizing in $S$.
Hence $\dim \coker \phi, \coker \psi < \dim R$ and  
the assertion follows.
\end{proof}

\begin{corollary}\label{cor: Frob rank equivalence}
Let $(R, \mf m, k)$ be an F-finite reduced local ring 
and $M$ be a finitely generated $R$-module. 
Then for $p^\gamma = p^{\dim R} [k^{1/p} : k]$
we have exact sequences
\[
F_* M \to \oplus^{p^\gamma} M \to T_1 \to 0 \text{ and }
\]
\[
\oplus^{p^\gamma} M \to F_* M \to  T_2 \to 0,
\]
where $\dim T_1, \dim T_2 < \dim R$. 
\end{corollary}
\begin{proof}
Let $\mf p \in \Minh (R)$, i.e., $\dim R/\mf p = \dim R$. Then $R_\mf p$ is a field, $k(\mf p)$, so 
$M_\mf p$ is a vector space over it. If $r$ is the rank of this vector space 
then we compute 
\[
\dim_{k(\mf p)} (F_* M)_\mf p = \dim_{k(\mf p)} F_* M_\mf p
= r \dim_{k(\mf p)} F_* k(\mf p) = r p^\gamma,
\]
by Theorem~\ref{thm: Kunz rank}. 
Therefore $\oplus^{p^\gamma} M_\mf p$ and $F_* M_\mf p$ are isomorphic 
for any $\mf p \in \Minh (R)$. 
\end{proof}

The following result is very much inspired by the treatment of Dutta's lemma (\cite{Dutta})
in \cite[Lemma~3.4]{HunekeSurvey}. 

\begin{lemma}\label{lem: measure defect}
Let $(R, \mf m, k)$ be an F-finite local ring with a residue field $k$
and let $M$ be a finitely generated $R$-module. 
Denote $p^{\dim M}[k : k^p] = p^{\gamma}$. 
Then there exists a finitely generated $R$-module $T$
such that $\Supp M = \Supp T$ and for all $\mf m$-primary ideals $I$ and 
$e \geq 0$ 
\[
\length (R/I \otimes_R F_*^e M) \leq 
p^{e\gamma} \length (T/IT).
\]
\end{lemma}
\begin{proof}
Note that for all $e \gg 0$ 
the annihilator of $F_*^e M$ is a radical ideal. 
We may replace $M$ with $F_*^{e_0} M$: if $T'$ is such that 
$\length (R/I \otimes_R F_*^{e} M) \leq 
p^{e \gamma} \length (T'/IT')$ 
for $e \geq e_0$, 
then the assertion holds for  $T = \oplus_{e = 0}^{e_0} F_*^e M \oplus T'$.
Hence by replacing $R$ by $R/\Ann M$, it suffices to assume that 
$R$ is reduced and $\Supp M = \Spec R$.

The proof proceeds by induction on $\dim M$. In the base case, $\dim M = 0$,
$M$ is a $k$-vector space. Let $m$ be its dimension. 
Then 
\[
\length (R/I \otimes_R F_*^{e} M)  = \length (F_*^{e} M) 
= m [k^{1/p^e} : k] =   m p^{e \gamma}
\]
and the  base case follows by taking $T = M$. 
 
In the general case, Corollary~\ref{cor: Frob rank equivalence}
provides a module $N$ such that $\dim N < \dim R$ and there is an exact sequence
$
\oplus^{p^\gamma} M \to F_*M \to N \to 0.
$
It follows that the inequality 
\[
\length (R/I \otimes_R F_*^{e + 1} M) \leq p^{\gamma} \length (R/I \otimes_R F_*^e M) + \length (R/I \otimes_R F_*^e N)
\]
holds for all $e \geq 0$. Adding up these inequalities gives the bound 
\begin{equation}\label{eq: defect}
\length (R/I \otimes F_*^e M) \leq 
p^{e\gamma} \length (M/IM) + 
\sum_{n = 0}^{e - 1} p^{(e-1-n)\gamma}\length (R/I \otimes F_*^n N).
\end{equation}
Since $\dim N \leq \dim R - 1$, by induction there is a module $T'$ such that
$\length (R/I \otimes F_*^n N) \leq p^{n(\gamma - 1)} \length (T'/IT')$, 
plugging this estimate in (\ref{eq: defect}) yields that 
\begin{align*}
\length (R/I \otimes F_*^e M) &\leq 
p^{e \gamma} \length (M/IM) + 
\sum_{n = 0}^{e - 1} p^{(e-1-n)\gamma}\length (R/I \otimes F_*^n N)
\\ &\leq 
p^{e\gamma} \length (M/IM) + 
\sum_{n = 0}^{e - 1} p^{e\gamma - n} \length (T'/IT')
\\ &\leq p^{e \gamma}\left(  \length (M/IM) + \length (T'/IT') \sum_{n= 0}^{e -1} p^{-n} \right) 
\\&\leq p^{e \gamma}\left(  \length (M/IM) + 2\length (T'/IT') \right).
\end{align*}
Since $\Spec R = \Supp M \subseteq \Supp (M \oplus T' \oplus T')$ 
the assertion follows for $T = M \oplus T' \oplus T'$. 
\end{proof}

\begin{corollary}\label{cor: intermediate convergence}
Let $(R, \mf m, k)$ be an F-finite local ring and $M$ be a finitely generated $R$-module. Denote $p^{\dim M}[k : k^p] = p^{\gamma}$. 
There exists a positive integer $e_0$ (which can be taken to be $0$ if $R$ is reduced) and a finitely generated $R$-module $D$ such that  $\dim D < \dim R$
and for any $\mf m$-primary ideal $I$ and all $e \geq 1$ 
\[
\left |\length (R/I \otimes_R F_*^{e + e_0} M)  - p^{e\gamma} \length (R/I \otimes_R F_*^{e_0} M)\right | 
\leq p^{e\gamma} \length (D/ID).
\]
\end{corollary}
\begin{proof}
There is no harm in assuming that $\dim M = \dim R$. 
Fix any $e_0$ such that $\sqrt{0}^{[p^{e_0}]} = 0$. Then $F_*^{e + e_0} M$ 
is an $\red{R}$-module for any $e\geq 0$. Hence Corollary~\ref{cor: Frob rank equivalence} applies and yields that 
\[
\left |\length (R/I \otimes_R F_*^{e +1 + e_0} M)  - p^{\gamma} \length (R/I \otimes_R F_*^{e + e_0} M)\right | \leq \length (R/I \otimes_R F_*^e T), 
\]
where $\dim T < \dim R$.
Adding these equations up bounds the difference
\[
\left |\length (R/I \otimes_R F_*^{e + e_0} M)  - p^{e\gamma} \length (R/I \otimes_R F_*^{ e_0} M)\right | \leq \sum_{i = 0}^{e - 1} p^{(e - i)\gamma}\length (R/I \otimes_R F_*^i T).
\]
Lemma~\ref{lem: measure defect}
provides $T'$ such that  
$\length (R/I \otimes_R F_*^i T) \leq p^{i (\gamma - 1)} \length (T'/IT')$ and $\dim T' = \dim T < \dim R$. 
Set $D = T' \oplus T'$. 
It follows that 
\[
\left |\length (R/I \otimes_R F_*^{e + e_0} M)  - p^{e\gamma} \length (R/I \otimes_R F_*^{ e_0} M)\right | \leq \sum_{i = 0}^{e - 1} \frac{p^{e\gamma - i}}{2}\length (D/ID)
\leq p^{e\gamma}\length (D/ID).
\]
\end{proof}

This corollary, combined with the computation in Lemma~\ref{lem: numgen}, immediately shows that the sequence 
$\length (R/\frpe{I})/p^{e\dim R}$, appearing in the definition of Hilbert--Kunz multiplicity, converges by the Cauchy criterion. 
Even more, by passing to the limit and using the standard field extension result, it leads to the following convergence estimate.

\begin{corollary}\label{cor: convergence formula}
Let $(R, \mf m, k)$ be a local ring of characteristic $p > 0$. 
There exists a faithfully flat local $R$-algebra $S$, such that $S/\mf mS$ is an algebraically closed field, a finitely generated $S$-module $T$ with $\dim T < \dim R$, and a non-negative integer $e_0$ (which can be taken to be $0$ if $R$ is F-finite and reduced) with the following property: for any $\mf m$-primary ideal $I$
\[
\left |\ehk (I)  - p^{- e_0 \dim R}  \length_R (R/I^{[p^{e_0}]})\right | 
\leq \length_S (T/IT).
\]
\end{corollary}
\begin{proof}
By Cohen's structure theorem, it is possible to write $\widehat{R}$ as a 
quotient of a power series ring $R = k[[x_1, \ldots, x_n]]/J$.
Denote $S = \overline{k} [[x_1, \ldots, x_n]]/J\overline{k} [[x_1, \ldots, x_n]]$.
By tensoring a composition series of a finite length $R$-module $M$ with $S$, it is easy to see that 
$\length_R (M) = \length_S (M \otimes_R S)$. 
In particular, $\ehk (IS) = \ehk (I)$ for any $\mf m$-primary ideal $I$. 
Hence, it suffices to replace $R$ with the F-finite ring $S$,
then the assertion follows from Corollary~\ref{cor: intermediate convergence} after dividing by $p^{(e + e_0)\gamma}$ and passing to the limit as $e \to \infty$.
\end{proof}

\subsection{Applications}

\subsubsection{Existence of F-signature} The estimate in Corollary~\ref{cor: convergence formula} provides uniform convergence of various families of ideals. 
This applies to F-signature as it can be expressed as the limit of a sequence of ideals 
that was introduced in \cite{AberbachEnescu, Yao}. 

\begin{lemma}\label{lem: splitting ideals}
Let $(R, \mf m, k)$ be an F-finite local domain of characteristic $p > 0$ and denote
\[
I_e = \left\{x \in R \mid \phi (F_*^e x) \in \mf m \text{ for all } \phi \in \Hom (F_*^e R, R) \right\}.
\]
Then 
\[
\frac{\length (R/I_e)}{p^{ed}} =
\frac{\max \{N \mid F_*^e R = \oplus^N R \oplus M\}}{\text{generic rank of } F_*^e R}.
\]

In particular, 
\[
\fsig(R) = \lim_{e \to \infty} \frac{\length (R/I_e)}{p^{ed}}.
\]
\end{lemma}
\begin{proof}
Let $a_e = a_e(R)$ and take a decomposition $F_*^e R = R^{\oplus a_e} \oplus M_e$. The key observation is that for any $\phi \in \Hom (F_*^e R, R)$ there is a containment $\phi (M_e) \subseteq \mf m$, otherwise there will be at least one more splitting. Therefore, $\mf m F_*^e R + M_e = \mf m^{\oplus a_e} \oplus M_e \subseteq F_*^e I_e$, 
so $R^{\oplus {a_e}} \otimes_R R/\mf m$ surjects onto $F_*^e (R/I_e)$ bounding its length from above. 
 On the other hand, the minimal generators of $R^{\oplus a_e}$ are clearly not in $F_*^e I_e$, so 
\[a_e = \length_R (F_*^e (R/I_e)) = [k^{1/p^e} : k] \length (R/I_e)\]
and the assertion follows from Theorem~\ref{thm: Kunz rank}.
\end{proof}

\begin{theorem}[Tucker, \cite{Tucker}]
The limit in the definition of $\fsig(R)$ exists. 
\end{theorem}
\begin{proof}
It is easy to see from the definition that $I_e^{[p]}\subseteq I_{e + 1}$, hence $\frpe{\mf m} \subseteq I_e$.  Furthermore, 
if $\mf m$ can be generated by $\mu$ elements, then 
$\mf m^{\mu p^e} \subseteq \mf m^{[p^e]}$ and 
it follows from Corollary~\ref{cor: convergence formula} that 
$\left |\ehk (I_e)  - \length (R/I_e)\right | \leq \length (T/\mf m^{\mu p^e} T)$. 
The existence of the Hilbert--Samuel polynomial, $n \mapsto \length (T/\mf m^nT)$ for $n \gg 0$, 
allows to find a constant $C$ such that the bound $\length (T/\mf m^{\mu p^e} T) < C p^{e \dim T}$ holds for all $e \gg 0$.
The proof is finished after noting that $\lim_{e \to \infty} \ehk (I_e)/p^{e\dim R}$ exists 
as this sequence is monotone due to the containment $I_e^{[p]} \subseteq I_{e + 1}$. 
\end{proof}

\subsubsection{The localization inequality for Hilbert--Kunz multiplicity as a swap of limits}
The main application needed in the next section is a formula presenting 
the Hilbert--Kunz multiplicity of $R_\mf p$ as a limit of $\mf m$-primary ideals. 
For this, consider a two-parameter family of ideals $(\frpe{I}, x^{np^e})$ where $I$ is an ideal such that $\dim R/I = 1$ and $x\in \mf m$ is such that $(I,x)$ is $\mf m$-primary. The first step  is to get control on the right-hand side of Corollary~\ref{cor: convergence formula}.

\begin{lemma}\label{lem: two parameter}
Let $(R, \mf m)$ be a local ring of characteristic $p > 0$, $I$ be an ideal
such that $\dim R/I = 1$,  and $M$ be a finitely generated $R$-module of dimension $d$. 
If $x \in \mf m$ is such that $(I, x)$ is $\mf m$-primary then 
there exists a constant $C$ 
such that $\length (M/(\frpe{I}, x^{np^e})M) < C p^{ed}n$. 
\end{lemma}
\begin{proof}
Since $M$ can be written as a quotient of a free $(R/\Ann M)$-module, 
it suffices to show the statement for $M = R$. 
In $A = R/\frpe{I}$ we have exact sequences
\[
0 \to A/(xA + 0:_A x^n) \to A/x^{n+1}A \to A/x^{n}A \to 0
\]
so $\length (A/x^{n+1}A) \leq \length (A/x^{n}A) + \length (A/xA)$. 
Hence, by induction  
\[
\length (R/(\frpe{I}, x^{np^e})R) \leq n \length (R/(\frpe{I}, x^{p^e})R) 
\]
and it remains to note that the Hilbert--Kunz function of $(I, x)$ is bounded by some
$Cp^{ed}$ as discussed before.
\end{proof}

If $x$ is an element and $I$ is an ideal of a ring, then, with a slight abuse of notation, 
I will use $xR/I$ to denote the ideal generated by the image of $x$ in $R/I$. 

\begin{theorem}\label{thm: Hilbert-Kunz descent}
Let $(R, \mf m)$ be a local ring of characteristic $p > 0$, $I$ be an ideal
such that $\dim R/I = 1$ and $\hght I = \dim R - 1$.
Then we have 
\[
\lim_{n \to \infty} \frac{\ehk ((I, x^n))}{n} = \sum_{\mf p \in \Lambda} \eh (xR/\mf p) \ehk (IR_\mf p),
\]
where the sum is taken over minimal primes $\mf p$ of $I$ such that $\dim R/\mf p = 1$ and $\dim R_\mf p = \dim R - 1$. 
\end{theorem}
\begin{proof}
Plugging Lemma~\ref{lem: two parameter} into Corollary~\ref{cor: convergence formula} gives a convergence estimate
\[
\left |\ehk((I, x^n)) - \frac{1}{p^{ed}}\length (R/(I^{[p^{e}]},x^{np^{e}})) \right | < 
n\frac{C}{p^{e - e_0}} \leq n\frac{C}{p^{e}}.
\]
Therefore, the bisequence $a_{n,e} =  \frac{1}{np^{ed}}\length (R/(\frpe{I},x^{np^e}))$
converges uniformly in $e$. 
Since the limits $\lim_{e \to \infty} a_{n,e}$ and $\lim_{n \to \infty} a_{n, e} = \eh (x^{p^e}R/\frpe{I})/p^{ed}$ 
exist, the standard result says that the double limits exist and are interchangeable. 
Hence, by  the associativity formula for multiplicity applied in $R/\frpe{I}$
\begin{align*}
\lim_{e \to \infty} \lim_{n \to \infty} a_{e, n} &=
\lim_{e \to \infty} \frac{1}{p^{e(d-1)}} \eh((x)R/\frpe{I})
= \lim_{e \to \infty} \frac{1}{p^{e(d-1)}} \sum_{\mf p \in \Minh (R/\frpe{I})} \eh (xR/\mf p) \length (R_\mf p/\frpe{I}R_\mf p) \\
&=\sum_{\mf p \in \Minh (R/I)} \eh (xR/\mf p) \lim_{e \to \infty} \frac{1}{p^{e(d-1)}} \length (R_\mf p/\frpe{I}R_\mf p) = 
\sum_{\mf p \in \Lambda} \eh (xR/\mf p) \ehk (IR_\mf p), 
\end{align*}
where it was first used that $\Minh (R/\frpe{I}) = \Minh (R/I)$ and, second, that  
this set can be cut down to $\Lambda$ because the limit was taken over $p^{e(d-1)}$.
\end{proof}

An easy consequence is a proof of the localization property for Hilbert--Kunz multiplicity, originally this property was observed in \cite[Corollary~3.8]{Kunz2}.

\begin{corollary}\label{cor: localization}
Let $(R, \mf m)$ be a local ring of characteristic $p > 0$, 
and $\mf p$ be a prime ideal such that 
$\dim R_\mf p + \dim R/\mf p = \dim R$. 
Then $\ehk (R_\mf p) \leq \ehk (R)$.
\end{corollary}
\begin{proof}
By localizing at an appropriate prime $\mf p \subsetneq \mf q \subseteq \mf m$,
it is enough to consider the case $\dim R/\mf p = 1$. 
Let $x$ be any element such that $(\mf p, x)$ is $\mf m$-primary. 
By Theorem~\ref{thm: Hilbert-Kunz descent}
\[
\lim_{n \to \infty} \frac{\ehk ((\mf p, x^n))}{n} = 
\eh (xR/\mf p) \ehk (R_\mf p).
\]
On the other hand, by (\ref{item: filtration}) of Proposition~\ref{prop: HK properties}
\[
\ehk ((\mf p, x^n)) \leq \length (R/(\mf p, x^n)) \ehk (R). 
\]
Thus, after taking the limit as $n \to \infty$ we obtain
$\eh (xR/\mf p) \ehk (R_\mf p) \leq \eh (xR/\mf p) \ehk (R)$. 
\end{proof}

\section{Equimultiplicity for Hilbert--Kunz multiplicity}\label{sec: Hilbert-Kunz}

I will describe a necessary condition for $\ehk$-equimultiplicity in the special case where $\dim R/\mf p = 1$. This case is easier to treat but still suffices for the applications. The ideas originate from \cite{SmirnovEqui}, however it was assumed there that $R/\mf p$ is regular as in Theorem~\ref{thm: multiplicity criterion}.

\subsection{A review of tight closure}
The characterization of equimultiplicity will be given in terms of tight closure, so let me list the needed facts from \cite{HochsterHuneke1}. A good introduction to the tight closure theory is \cite{Hochster}.

\begin{theorem}\label{thm: tc review}
Let $R$ be a Noetherian ring of characteristic $p > 0$.
\begin{enumerate}
\item $x \in I^*$ if and only if $cx^{p^e} \in \frpe{I}$ for all $e \gg 0$ and $c\in R$ which is not contained in any minimal prime. 
\item If $R$ is complete and reduced, then by \cite{HHSmooth} 
it has a \emph{test element}: 
an element $c$ such that $c$ is not contained in any minimal prime and the containment $cx^{p^e} \in \frpe{I}$ holds for all $I$, all $x \in I^*$, and all $e \geq 1$. 
Test elements exist if $R$ is complete and reduced by  \cite{HHSmooth}.
\item If $R$ is formally unmixed (i.e., $\Ass (\widehat{R}) = \Minh (\widehat{R})$) then for a pair of $\mf m$-primary ideals 
$I \subseteq J$ we have $\ehk (I) = \ehk (J)$ if and only if $J \subseteq I^*$.
\end{enumerate}
\end{theorem}

\begin{lemma}\label{lemma: tc intersect}
Suppose that $(R, \mf m)$ has a test element $c$. 
Let $I$ be an ideal and $x \in \mf m$ be an element. 
Then 
\[
\cap_{n \geq 1} (I + x^n)^* = I^*
\]
\end{lemma}
\begin{proof}
Take an element $z$ in the intersection on the left. Since $c$ is a test element,
$cz^{p^e} \in (\frpe{I}, x^{np^e})$ for all $e, n \geq 1$. 
But $\frpe{I} = \cap_n (\frpe{I}, x^{np^e})$ by Krull's intersection theorem in $R/\frpe{I}$, so it follows that $z \in I^*$. The opposite containment is clear. 
\end{proof}

\begin{lemma}\label{lemma: filtration in tight closure}
Let $(R, \mf m)$ be a formally unmixed local ring of dimension $d$ and characteristic $p > 0$ and $I$ be an $\mf m$-primary ideal. 
Suppose that $L_e$ is a sequence of ideals such that 
\begin{enumerate}
\item $\frpe{I} \subseteq L_e$, 
\item $L_e^{[p]} \subseteq L_{e + 1}$,
\item $\lim_{e \to \infty} \length (R/L_e)/p^{ed} = \ehk (I)$.
\end{enumerate}
Then $L_e \subseteq (\frpe{I})^*$. 
\end{lemma}
\begin{proof}
The assumptions give containments $I^{[p^{e + e'}]} \subseteq L_e^{[p^{e'}]} \subseteq L_{e + e'}$.
Then it is easy to see that  $\ehk (L_e) = \ehk (\frpe{I})$ and the assertion follows from Theorem~\ref{thm: tc review}. 
\end{proof}

There is also a converse. 

\begin{lemma}\label{lemma: tc limit}
Let $(R, \mf m)$ be a local ring of characteristic $p > 0$ and $I$ be an $\mf m$-primary ideal. 
Let $I_e$ be a sequence of ideals such that $I^{[p^e]} \subseteq I_e \subseteq (I^{[p^e]})^*$.
If $R$ has a test element $c$, then
\[
\lim\limits_{e \to \infty} \frac{1}{p^{e\dim R}} \length(R/I_e) = \ehk (I).
\]
\end{lemma}
\begin{proof}
We tensor the exact sequence $R \xrightarrow{c} R \to R/(c) \to 0$ with $R/I_e$. 
Since $cI_e \subseteq c(\frq{I})^* \subseteq \frq{I}$, it is immediate that the sequence
\[
R/I_e  \xrightarrow{c} R/\frq{I} \to R/(c, \frq{I}) \to 0
\]
is still exact. From the sequence once gets the estimate 
\[
\length (R/I_e) \leq \length (R/\frq{I}) \leq \length (R/I_e) + \length (R/(c, \frq{I})R).
\]
Note that $\dim R/cR < \dim R$ since $c$ is not contained in any minimal prime, so 
the estimate gives that
$0 =  
\lim\limits_{e \to \infty} \frac{1}{p^{e\dim R}} \left ( \length(R/I_e) - \length (R/\frq{I}) \right) $ and the assertion follows.
\end{proof}

\subsection{Equimultiplicity}
As a starting point, recall that Theorem~\ref{thm: Hilbert-Kunz descent} gives a descent formula  for Hilbert--Kunz multiplicity: if $\dim R/\mf p = 1$ and $\dim R_\mf p = \dim R - 1$ then
\[
\lim_{n \to \infty} \frac{\ehk ((\mf p, x^n))}{n} =  \eh (xR/\mf p) \ehk (R_\mf p).
\]
A key observation is that the sequence in the left-hand side is non-increasing.

\begin{lemma}\label{lem: HK filtration}
Suppose that $\dim R/\mf p = 1$ and $x \notin \mf p$. Then 
\[
\frac{\ehk ((\mf p, x^n))}{n}  
\geq \frac{\ehk ((\mf p, x^{n+1}))}{n+1} 
\]
Moreover, the equality holds if and only if 
\[
\ehk ((\mf p, x^n)) = n \lim_{e \to \infty} \frac{\length (R/(x^{p^e}R + \mf p^{[p^e]}:_Rx^{np^e}))}{p^{e\dim R}}.
\]
\end{lemma}
\begin{proof}
First, take $I$ to be any ideal such that $\dim R/I = 1$ and  the image of $x$ is a parameter in $R/I$. We have the exact sequence 
\begin{equation}\label{eq: filtration sequence}
0 \to R/(xR + I:_R x^k) \to R/(I, x^{k+1}) \to R/(I, x^k) \to 0.
\end{equation}
These exact sequences can be iterated to deduce that 
\begin{equation}\label{eq: filter formula}
\length (R/(I, x^n)) = \length (R/(I, x^{n-1}) + \length (R/(xR + I:_R x^{n-1}))
= \sum_{i = 0}^{n-1} \length (R/(xR + I:_Rx^i)).
\end{equation}
It immediately follows that $\length (R/(I, x^n)) \geq 
n \length (R/(xR + I:_Rx^{n}))$ which allows to deduce from (\ref{eq: filtration sequence})  the estimate
\[
\length (R/(I, x^{n+1})) =  \length (R/(I, x^n)) + \length (R/(xR + I:_Rx^{n}))
\leq \length (R/(I, x^n)) + \frac{1}{n} \length (R/(I, x^n)). 
\]
Plugging $I = \mf p^{[p^e]}$ and replacing $x$ by $x^{p^e}$ yields the estimate
\begin{align*}
\length (R/(\mf p^{[p^e]}, x^{p^e(n+1)}))  &=  \length (R/(\mf p^{[p^e]}, x^{np^e}))  + \length (R/(x^{p^e}R + \mf p^{[p^e]}:_Rx^{np^e}))
\\ &\leq \frac{n+1}{n}  \length (R/(\mf p^{[p^e]}, x^{np^e})) 
\end{align*}
which implies both assertions after dividing by $p^{e\dim R}$ and taking the limits.
\end{proof}

\begin{remark}\label{rmk: multiplicity property}
Before the next proof let me remind some properties of multiplicity. In what follows $S$ is a one-dimensional local ring $S$ and $x$ is a parameter.  
First, set $I = 0$ in (\ref{eq: filter formula}) and note that $\length (S/(xS + 0:_S x^i)) \leq \length (S/xS)$. It is then immediate that 
$\eh(xS) \leq \length (S/xS)$ and that the 
equality holds if and only if $0 :_S x^i = 0$ for all $i$, \emph{i.e.}, $S$ is Cohen-Macaulay. 

In the following proof we will be given $S = R/I$, in which case $\eh (xS) = \length (S/xS)$ if and only if $\mf m$ is not an associated prime of $I$, \emph{i.e.}, $I :_R \mf m^\infty \coloneqq \cup_n I:_R \mf m^n = I$. 
\end{remark}

\begin{theorem}\label{thm: equimultiple property}
Let $(R, \mf m)$ be a formally unmixed local ring of characteristic $p > 0$ and dimension $d$, $c \in $ be a test element, and 
$\mf p$ be a prime ideal such that $\dim R/\mf p = 1$ and $\dim R_\mf p = d - 1$.
If $\ehk (R) = \ehk (R_\mf p)$ then 
$(\mf p^{[p^e]})^* :_R \mf m^{\infty} = (\mf p^{[p^e]})^*$ for all integers $e \geq 0$. 
\end{theorem}
\begin{proof}
Let $x \in \mf m\setminus \mf p$ be an arbitrary element. 
Note that by (\ref{item: filtration}) of Proposition~\ref{prop: HK properties}
\[\ehk ((\mf p, x)) \leq \ehk (R) \length (R/\mf p+ xR) = 
\ehk (R)\eh (xR/\mf p)\] where the last equality holds by Remark~\ref{rmk: multiplicity property}. 
Thus by Lemma~\ref{lem: HK filtration} for any positive integer $n$ there is a bound
\begin{equation}\label{eq: equimultiplicity inequalities}
\ehk (R)\eh (xR/\mf p) \geq \ehk ((\mf p, x)) \geq  \frac{\ehk ((\mf p, x^n))}{n} \geq 
\frac{\ehk ((\mf p, x^{n+1}))}{n+1} \geq
\eh (xR/\mf p) \ehk (R_\mf p).
\end{equation}
Since $\ehk (R) = \ehk (R_\mf p)$ there should be equality throughout, 
so Lemma~\ref{lem: HK filtration} gives that for any $n$ 
\begin{equation}\label{eq: colons HK}
\ehk ((\mf p, x)) = \frac{1}{n}\ehk ((\mf p, x^n)) = \lim_{e \to \infty} \frac{\length (R/(x^{p^e}R + \mf p^{[p^e]}:_Rx^{np^e}))}{p^{e\dim R}}.
\end{equation}

It is then easy to check that for any $n$ the sequence of ideals 
$L_e = x^{p^e}R + \mf p^{[p^e]}:_Rx^{np^e}$ satisfies 
the assumptions of Lemma~\ref{lemma: filtration in tight closure} for $I = (\mf p, x)$. 
Therefore, for all $e$ and $n$
\[
(x^{p^e}R + \mf p^{[p^e]}:_Rx^{np^e}) \subseteq (x^{p^e}R + \mf p^{[p^e]})^*.
\]
After letting $n$ grow, it follows that  $\mf p^{[p^e]}:_R \mf m^{\infty} = \mf p^{[p^e]}:_R x^{\infty} \subseteq (x^{p^e}R + \mf p^{[p^e]})^*$. 
Because $x$ was chosen arbitrarily, we might have taken instead any power $x^k$ in the argument above, so by Lemma~\ref{lemma: tc intersect} for any positive integer $e$ there is a containment
\[
\mf p^{[p^e]}:_R \mf m^{\infty} \subseteq \cap_n (x^{np^e}R + \mf p^{[p^e]})^*
= (\mf p^{[p^e]})^*.
\]

Now I claim that $(\mf p^{[p^e]})^* : x \subseteq (\mf p^{[p^e]})^*$ for any element $x \in \mf m\setminus \mf p$. If not, there is $a \in \mf m$ such that $ax \in (\mf p^{[p^e]})^*$, i.e., 
$ca^{p^{e'}} x^{p^{e'}} \in \mf p^{[p^{e + e'}]}$ for all $e' \gg 0$. 
Hence, 
$ca^{p^{e'}}  \in \mf p^{[p^{e + e'}]} : \mf m^\infty \subseteq 
(\mf p^{[p^{e + e'}]})^*$, hence $c^2 a^{p^{e'}} \subseteq \mf p^{[p^{e + e'}]}$ 
proving that $a \in (\mf p^{[p^e]})^*$. As we may again replace $x$ with any power $x^k$, $(\mf p^{[p^e]})^* : x^k \subseteq (\mf p^{[p^e]})^*$ for all $k$
and the assertion follows.
\end{proof}

\begin{corollary}\label{cor: colength and multiplicity}
Let $(R, \mf m)$ be a formally unmixed local ring of characteristic $p > 0$ and dimension $d$, $c$ be a test element, and 
$\mf p$ be a prime ideal such that $\dim R/\mf p = 1$ and $\dim R_\mf p = d - 1$.
If $\ehk (R) = \ehk (R_\mf p)$ then 
for any $x \in \mf m \setminus \mf p$ we have equality
\[ 
\length (R/((\mf p^{[p^e]})^* + xR))
= \eh (xR/(\mf p^{[p^e]})^*)
= \eh (xR/\mf p) \length (R_\mf p/(\mf p^{[p^e]})^*R_\mf p).
\]
\end{corollary}
\begin{proof}
Because $\mf m$ is not an associated prime of $(\mf p^{[p^e]})^*$ by Theorem~\ref{thm: equimultiple property}, the quotient ring $S = R/(\mf p^{[p^e]})^*$ is a one-dimensional Cohen-Macaulay local ring with the only minimal prime $\mf p S$.
Thus 
Remark~\ref{rmk: multiplicity property} and the associativity formula for the multiplicity  
immediately demonstrate that 
\[\length (S/xS) = \eh (xS) = \eh (xS/\mf p)\length (S_\mf p).
\]
\end{proof}

In the case where $R/\mf p$ is regular, the tight closure condition is not just necessary but is also sufficient.

\begin{corollary}
Let $(R, \mf m)$ be a formally unmixed local ring of characteristic $p > 0$ and dimension $d$, $c$ be a test element, and 
$\mf p$ be a prime ideal such that $R/\mf p$ is a one-dimensional regular ring 
and $\dim R_\mf p = d - 1$.
Then $\ehk (R) = \ehk (R_\mf p)$ if and only if 
$(\mf p^{[p^e]})^* :_R \mf m^{\infty} = (\mf p^{[p^e]})^*$ for all non-negative integers $e$. 
\end{corollary}
\begin{proof}
The goal is to show the converse to Theorem~\ref{thm: equimultiple property}.
By the assumption, there is an element $x \in R$ such that $\mf m = \mf p + xR$ 
and the image of $x$ is a regular element in all $R/(\mf p^{[p^e]})^*$, 
a one-dimensional Cohen--Macaulay ring. Hence by Lemma~\ref{lemma: tc limit}
\begin{align*}
\ehk (\mf m) &= \lim_{e \to \infty} \frac{\length (R/((\mf p^{[p^e]})^* + x^{p^e}R))}{p^{ed}}
= \lim_{e \to \infty} \frac{\eh (x^{p^e}R/(\mf p^{[p^e]})^*)}{p^{ed}}
\\&= \lim_{e \to \infty} \frac{\eh (x^{p^e}R/\mf p) \length (R_\mf p/(\mf p^{[p^e]})^*R_\mf p)}{p^{ed}} = \lim_{e \to \infty} \frac{\length (R_\mf p/(\mf p^{[p^e]})^*R_\mf p)}{p^{e(d-1)}} = \ehk (R_\mf p).
\end{align*}
\end{proof}

While it looks different, this formula is a positive characteristic analogue
of Theorem~\ref{thm: multiplicity criterion}.
Namely, by the characterization of the analytic spread in terms 
of integral closure given by Burch in \cite{Burch},  Theorem~\ref{thm: multiplicity criterion} asserts that if $R/\mf p$ is a DVR then
 $\eh (R_\mf p) = \eh (R)$ if and only if $\overline{\mf p^n}$ is $\mf p$-primary for all $n$.

\begin{remark}
In the case where $\dim R/\mf p > 1$, the characterization is slightly more complicated. Loosely speaking, if $\mf p$ is an $\ehk$-equimultiple prime then $\mf p^{[p^e]}$ should be ``Cohen-Macaulay up to tight closure'': if $m = \dim R/\mf p$, then there exist elements $x_1, \ldots, x_m$ such that $x_{i + 1}$ is not a zerodivisor modulo $(\mf p^{[p^e]}, x_1^{p^e}, \ldots, x_i^{p^e})^*$ for all $e$ and all $i < m$. 
See \cite{SmirnovEqui}.
\end{remark}

\subsection{Applications}

\subsubsection{Rigidity in weakly F-regular rings}
Theorem~\ref{theorem: F-signature is rigid} shows that F-signature is rigid. 
A similar result holds for Hilbert--Kunz multiplicity provided that tight closure is a trivial operation. 

\begin{definition}
We say that a ring $R$ is weakly F-regular if $I^* = I$ for all ideals $I$. 
\end{definition}

\begin{proposition}[Rigidity]
Let $(R, \mf m)$ be weakly F-regular local ring and $\mf p$ be  a prime ideal
such that $\dim R/\mf p + \dim R_\mf p = \dim R$. 
Then $\ehk (R) = \ehk (R_\mf p)$ if and only if 
$\length (R/\mf m^{[p^e]}) = p^{e \dim R/\mf p} \length (R_\mf p/\mf p^{[p^e]}R_\mf p)$ 
for all $e \geq 1$. 

In particular, if $R$ is a weakly F-regular F-finite domain 
then $\ehk (R) = \ehk (R_\mf p)$ if and only if the numbers of minimal generators of 
$F_*^e R$ and  $F_*^e R_\mf p$ are equal for all $e \geq 1$. 
\end{proposition}
\begin{proof}
Suppose that $\ehk (R) = \ehk (R_\mf p)$ and use induction on $\dim R/\mf p$ 
to reduce to the case where $\dim R/\mf p = 1$. Let $x \in \mf m$ be such that $\mf p + xR$ is $\mf m$-primary. By Corollary~\ref{cor: colength and multiplicity} for all $e \geq 1$
\[
\length (R/(\mf p^{[p^e]} + x^{p^e}R))
= p^e \eh (xR/\mf p) \length (R_\mf p/\mf p^{[p^e]}R_\mf p).
\]
In particular, $\ehk (\mf p + xR) = \eh (xR/\mf p) \ehk (R_\mf p) = \eh (xR/\mf p) \ehk (\mf m)$.

Now, proceed as in the proof of Proposition~\ref{prop: HK properties}
and take a composition series of $R/(\mf p + xR)$: 
\[I_0 = \mf p + xR \subsetneq I_1 \subsetneq \cdots \subsetneq I_\ell = R,\] where $I_{k+1} = (I_k, u_k)$ with $\mf mu_k \subseteq I_k$. 
Raising to Frobenius powers  gives a filtration  $\mf p^{[p^e]} + x^{p^e}R \subsetneq I_1^{[p^e]} \subsetneq \cdots$ which computes that
\[
\length (R/(\mf p^{[p^e]} + x^{p^e}R)) = \sum_{k = 0}^{\ell - 1} \length (R/I_k^{[p^e]} : u_k^{p^e}) \leq \length (R/(\mf p + xR)) \length (R/\mf m^{[p^e]}). 
\]
Since the limits of the left-hand side and the right-hand side agree, 
we must have that
\[
\lim_{e \to \infty} \frac{\length (R/I_k^{[p^e]} : u_k^{p^e})}{p^{e \dim R}}
= \ehk (\mf m) \text { for all } 0 \leq k \leq \ell - 1.
\]
By Lemma~\ref{lemma: filtration in tight closure} 
$\mf m^{[p^e]} \subseteq I_k^{[p^e]} : u_k^{p^e} \subseteq (\mf m^{[p^e]})^* = \mf m^{[p^e]}$, hence the equality holds and it follows that 
\[
p^e \eh (xR/\mf p) \length (R_\mf p/\mf p^{[p^e]}R_\mf p) = 
\length (R/(\mf p^{[p^e]} + x^{p^e}R)) =  \length (R/(\mf p + xR)) \length (R/\mf m^{[p^e]})
\]
and one direction follows. The converse is clear. 
\end{proof}

\subsubsection{Hilbert--Kunz multiplicity and singularity}

I will now present a new proof of  the celebrated theorem of Watanabe and Yoshida
which opened the use of  Hilbert--Kunz theory in the study of singularities.

\begin{theorem}[Watanabe--Yoshida]
Let $(R, \mf m)$ be a local ring of characteristic $p > 0$. 
Then $R$ is regular if and only if $\ehk (R) = 1$ and $R$ is  formally unmixed (i.e., $\Ass \widehat{R} = \Minh (\widehat{R})$).
\end{theorem}

It seems that all proofs of this statement pass through the following crucial lemma
which reduces the assertion to the Kunz theorem.

\begin{lemma}[Watanabe--Yoshida]\label{lem: WY condition}
Let $(R, \mf m)$ be a local ring such that $\ehk (R) = 1$. 
Suppose that there exists an ideal $I \subseteq \mf m^{[p]}$
such that $\ehk (I) \geq \length (R/I)$. 
Then $R$ is regular.
\end{lemma}
\begin{proof}
By (\ref{item: filtration 2}) of Proposition~\ref{prop: HK properties},
$\ehk (I) \leq \length (\mf m^{[p]}/I) \ehk (R) + \ehk (\mf m^{[p]})
= \length (\mf m^{[p]}/I)  + p^d$. 
Therefore, 
\[
 \length (\mf m^{[p]}/I)  + p^d \geq \ehk (I) \geq \length (R/I) = \length (\mf m^{[p]}/I) + \length (R/\mf m^{[p]}) 
\]
and we obtain that $p^d \geq \length (R/\mf m^{[p]})$. This implies that $R$ is regular by Theorem~\ref{theorem Kunz}.
\end{proof}

Thus the Watanabe--Yoshida theorem amounts to finding an ideal $I$ fitting the assumptions of the lemma. There are several possible options. The original proof of Watanabe and Yoshida
uses a \emph{parameter} ideal, but the inequality $\ehk (I) \geq \length (R/I)$ requires $R$ to be Cohen-Macaulay and this is hard to show.
Second, in \cite{MQSLength2} we showed that $\ehk (I) \geq \length (R/I)$
holds for any \emph{integrally closed} ideal, so one can take for example 
$I = \overline{\mf m^n}$ for any $n \gg 0$. 
The third construction, due to Huneke and Yao in \cite{HunekeYao}, is the simplest of all proofs that I know. In this case,  $I = \mf p^{(n)} + xR$ where 
$\mf p$ is a prime such that $\dim R/\mf p = 1$ and $R_\mf p$ is regular, $x \in \mf m^{[p]} \setminus \mf p$, and $n$ is taken sufficiently large for 
$I$ to be contained in $\mf m^{[p]}$, such $n$ exists by Chevalley's lemma. Such primes exist:  if $\mf p$ is such that $\dim R/\mf p = 1$ and $\dim R_\mf p = \dim R - 1$, then a) $\ehk (R_\mf p) = \ehk (R)$ by the localization property, b) $R_\mf p$ is regular by induction on $\dim R$.

The proof below also starts with a prime $\mf p$ with the above properties, but it uses $\ehk$-equimultiplicity to construct the ideal $I$. 
This approach conceptualizes the construction of Huneke and Yao,  although it requires a lot more work to build the equimultiplicity machinery.

Last, I want to note that if $R$ is strongly F-regular, then one can avoid the Watanabe--Yoshida lemma and modify the proof of Theorem~\ref{theorem signature 1 iff regular} to get an easy proof, see \cite{PolstraSmirnov}. However, there does not seem to be an easy way for showing that a complete domain with $\ehk (R) = 1$
is strongly F-regular.

\begin{proof}[The proof of the Watanabe--Yoshida theorem]
If $R$ is regular, then $\ehk (R) = 1$ by Theorem~\ref{theorem Kunz}, the main direction  is the converse. The proof is by induction on $\dim R$. The base case $\dim R = 0$ is trivial 
since $1 = \ehk (R) = \length (R)$, so $R$ is a field. 

It is enough to assume that $R$ is complete, as the assumptions pass to the completion. It also suffices to assume that $R$ is reduced.
Namely,  $\ehk (R) \geq \ehk (\red{R}) \geq 1$, so $\ehk (\red{R}) = 1$. 
Now, if $\red{R}$ is regular, then $\mf p = \sqrt{0}$ is prime and, by the associativity formula, 
\[1 = \ehk (R) = \ehk (R/\mf p) \length (R_\mf p) = \length(R_\mf p).\] 
Hence, $R_\mf p$ is a field. Because $\Ass (R)= \{\mf p\}$,  $R\setminus \mf p$ consists of non zerodivisors, so $\mf p = 0$ and $R = \red{R}$ is regular. 

It is known by \cite{HHSmooth} that a complete reduced ring has a test element. 
As explained before the proof, there is a prime ideal $\mf p$ such that 
$\dim R/\mf p = 1$ and $\ehk (R_\mf p) = 1$, so $R_\mf p$ is regular by induction.

\begin{claim}
For $L_e = (\mf p^{[p^e]})^* + x^{p^e}R$ we have 
$\length (R/L_e) = \ehk (L_e)$.
\end{claim}
\begin{proof}
First, by Corollary~\ref{cor: colength and multiplicity}
\[
\length (R/((\mf p^{[p^e]})^* + x^{p^e}R)) = \eh (x^{p^e}R/(\mf p^{[p^e]})^*)
= p^e\eh (xR/\mf p) \length (R_\mf p/(\mf p^{[p^e]})^*R_\mf p).
\]
Since $R_\mf p$ is regular, we compute  $\length (R_\mf p/(\mf p^{[p^e]})^*R_\mf p) =
\length (R_\mf p/\mf p^{[p^e]}R_\mf p) = p^{e\dim R_\mf p}$.
Therefore  
\[
\length (R/((\mf p^{[p^e]})^* + x^{p^e}R)) = p^{e\dim R} \eh (xR/\mf p).
\]

Second, $\ehk ((\mf p^{[p^e]})^* + x^{p^e}R) = \ehk ((\mf p+ xR)^{[p^e]}) = p^{e\dim R} \ehk (\mf p + xR)$, because $(\mf p^{[p^e]})^* + x^{p^e}R$ is contained in the tight closure of $\frpe{(\mf p + xR)}$. 
But then the inequalities (\ref{eq: equimultiplicity inequalities}) in the proof of Theorem~\ref{thm: equimultiple property} show that 
$\ehk (\mf p + xR)) = \ehk (xR/\mf p)\ehk(R_\mf p) = \ehk (xR/\mf p)$ and the claim follows.
\end{proof}

It remains to show that $L_e \subseteq \mf m^{[p]}$ for a large $e$.
This follows from Chevalley's lemma or from the properties of integral closure: there a constant $c$ such that  $L_e \subseteq \overline{(\mf p, x)^e} \subseteq (\mf p, x)^{e - c}$  because $R$ is a complete domain, see \cite[Proposition~5.3.4]{HunekeSwanson}. The assertion now follows from Lemma~\ref{lem: WY condition}.
\end{proof}

\subsubsection{An application to semicontinuity.}

The study of semicontinuity in the theory of F-invariants originates 
from \cite{Kunz2} and \cite{EnescuShimomoto}. It is known by \cite{SmirnovSemi, Lyu} that Hilbert--Kunz multiplicity is an upper semicontinuous function on the spectrum of an excellent ring, see also \cite{Polstra, PolstraTucker} for F-signature, and \cite{SPYGlobal} for Frobenius Betti numbers. However, it was originally believed that Hilbert--Kunz multiplicity would be strong upper semicontinuous just as multiplicity. I will now present an argument from \cite{SmirnovEqui} that shows that this is not the case. The counter-example is based on Monsky's computation of the Hilbert--Kunz multiplicity of a family of quartics in $3$ variables in \cite{MonskyQP}.

\begin{theorem}[Monsky, [\cite{MonskyQP}]]\label{thm: Monsky computation}
Let $K$ be an algebraically closed field of characteristic $2$. 
For $\alpha \in K$ let $Q_\alpha = K[[x,y,z]]/(z^4 +xyz^2 + (x^3 + y^3)z + \alpha x^2y^2)$.
Then
\begin{enumerate}
\item $\ehk (Q_\alpha) = 3 + \frac 12$, if $\alpha = 0$,
\item $\ehk (Q_\alpha) = 3 + 4^{-m}$, 
if $\alpha \neq 0$ is algebraic over $\mathbb Z/2\mathbb Z$, where 
$m = [\mathbb Z/2\mathbb Z(\lambda): \mathbb Z/2\mathbb Z]$ for $\lambda$ such that $\alpha = \lambda^2 + \lambda$
\item $\ehk (Q_\alpha) = 3$ if $\alpha$ is transcendental over $\mathbb Z/2\mathbb Z$.
\end{enumerate} 
\end{theorem}
\begin{proof}
The last two cases are computed by Monsky in \cite{MonskyQP}.
In the first case there is a factorization 
\[
z^4 +xyz^2 + (x^3 + y^3)z = z (x + y + z) (x^2 + y^2 + z^2 + zx + zy + xy)
= z(x + y + z) ((x + y + z)^2 + zy).
\]
Hence, by the associativity formula
\begin{align*}
\ehk(Q_0) &= \ehk (K[x,y]) + \ehk(K[x,y,z]/(x + y + z)) + \ehk(K[x, y, z]/((x + y + z)^2 + zy)) \\
&= 2 + \ehk(K[x, y, z]/(x^2 + zy)) = 3 \frac{1}{2},
\end{align*}
as the last equation, $(A_1)$-singularity, was computed in 
\cite[Theorem~5.4]{WatanabeYoshida}.
\end{proof}

Building on this computation, in \cite{BrennerMonsky}
Brenner and Monsky showed that tight closure does not commute 
with localization on the hypersurface\footnote{
Monsky showed that quartics split into two families and studied their Hilbert--Kunz multiplicities in 
\cite{MonskyQL, MonskyQP}. The hypersurface of Brenner and Monsky parametrizes the family appearing in \cite{MonskyQP}. Yet, a recent work \cite{Simpson} establishes the failure of localization in the second family \cite{MonskyQL}.}
\[R = F[x, y, z, t]/(z^4 +xyz^2 + (x^3 + y^3)z + tx^2y^2),\] 
where $F$ is an algebraic closure of $\mathbb Z/2\mathbb Z$. 
By the Jacobian criterion, the singular locus of this hypersurface is $V((x,y,z))$ 
and $\mf p = (x,y, z)$ is prime, in fact $R/\mf p \cong F[t]$.

\begin{proposition}\label{HK counterexample}
Let $F$ be the algebraic closure of $\mathbb Z/2\mathbb Z$, 
$R = F[x, y, z, t]/(z^4 +xyz^2 + (x^3 + y^3)z + tx^2y^2)$, and $\mf p = (x,y ,z)$, a prime ideal of $R$.
Then  $\ehk (\mf p) = 3$, but $\ehk(\mf m) > 3$ for any maximal ideal $\mf m$ containing $\mf p$.
\end{proposition}
\begin{proof}
The residue field of $R_\mf p$ is $K = F(t)$, so if we define $Q_t$ as in Theorem~\ref{thm: Monsky computation}, then
$\widehat{R_\mf p} \cong Q_t$ by Cohen's structure theorem. Thus $\ehk(R_\mf p) = 3$. Since $F$ is algebraically closed, all maximal ideals containing $\mf p$ are of the form $\mf m_\alpha = (\mf p, t - \alpha)$ for $\alpha \in F$.
By the way of contradiction, suppose that $\ehk (\mf m_\alpha) = \ehk (\mf p)$. Then 
by Corollary~\ref{cor: colength and multiplicity} 
\[
\length (R_{\mf m_\alpha}/(t-\alpha, (\mf p^{[p^e]})^*))
= \length (R_\mf p/(\mf p^{[p^e]})^*R_\mf p)
\]
for all $e \geq 1$. Now, Lemma~\ref{lemma: tc limit}
computes that
\[
\ehk (R_{\mf m_\alpha}/(t - \alpha)R_{\mf m_\alpha}) = 
\lim_{e \to \infty} \frac{\length (R_{\mf m_\alpha}/(t-\alpha, (\mf p^{[p^e]})^*))}{p^{2e}}
= \lim_{e \to \infty} \frac{\length (R_\mf p/(\mf p^{[p^e]})^*R_\mf p)}{p^{2e}}
= 3
\]
contradicting Theorem~\ref{thm: Monsky computation} where it was computed that $\ehk (R_{\mf m_\alpha}/(t - \alpha)R_{\mf m_\alpha}) > 3$.
\end{proof}

\begin{corollary}\label{not locally constant}
Let $R = F[x, y, z, t]/(z^4 +xyz^2 + (x^3 + y^3)z + tx^2y^2)$ where $F$ is a field of characteristic $2$. Then the set $\{\mf q \in \Spec R \mid \ehk(\mf q) \leq 3\}$ is not open. In particular, $\ehk$ is not strongly upper semicontinuous.
\end{corollary}
\begin{proof}
If the set was open, its intersection with $V((x,y,z))$ should be open and non-empty. 
In particular, all but finitely many maximal ideals $\mf m$ containing $(x,y,z)$ should belong to the open set.
\end{proof}

\section{Further directions}

There are many other invariants derived from Frobenius for which one could study  equimultiplicity. In my opinion, the most promising are dual F-signature (\cite{HochsterYao, Sannai, SmirnovTucker}), the Frobenius--Betti numbers (\cite{AberbachLi, DHNB}), and the Frobenius--Euler characteristics (\cite{SPYGlobal})
because they are already known to be semicontinuous. 

In addition, our understanding of equimultiplicity of  F-signature is still unsatisfactory, 
because there is no intrinsic characterization similar to Theorem~\ref{thm: equimultiple property}. In particular, the strong semicontinuity of other F-invariants still might possible, although unlikely. 

Second, we have no analog of Theorem~\ref{thm: Dade blowup} as  
it is known (see Section~5 of \cite{MPST}) that F-signature and Hilbert--Kunz multiplicity do not behave well after blowing up the maximal 
ideal of an isolated singularity. However, it would be naive to expect 
that Hilbert--Kunz multiplicity will suffice alone, the resolving invariant 
in characteristic $0$ is far more complex. 

Last, there is a related notion of equimultiplicity for singularities in families, see 
 \cite{Lipman}. This side has not been touched much for F-invariants, 
it is only known that Hilbert--Kunz multiplicity is semicontinuous \cite{Affine}.

%

\bibliographystyle{plain}
\bibliography{refs}

\begin{thebibliography}{10}

\bibitem{AberbachEnescu}
Ian~M. Aberbach and Florian Enescu.
\newblock The structure of {F}-pure rings.
\newblock {\em Math. Z.}, 250(4):791--806, 2005.

\bibitem{AberbachLeuschke}
Ian~M. Aberbach and Graham~J. Leuschke.
\newblock The {$F$}-signature and strong {$F$}-regularity.
\newblock {\em Math. Res. Lett.}, 10(1):51--56, 2003.

\bibitem{AberbachLi}
Ian~M. Aberbach and Jinjia Li.
\newblock Asymptotic vanishing conditions which force regularity in local rings
  of prime characteristic.
\newblock {\em Math. Res. Lett.}, 15(4):815--820, 2008.

\bibitem{Bennett}
Bruce~Michael Bennett.
\newblock On the characteristic functions of a local ring.
\newblock {\em Ann. of Math. (2)}, 91:25--87, 1970.

\bibitem{BlickleSchwedeTucker}
Manuel Blickle, Karl Schwede, and Kevin Tucker.
\newblock {$F$}-signature of pairs and the asymptotic behavior of {F}robenius
  splittings.
\newblock {\em Adv. Math.}, 231(6):3232--3258, 2012.

\bibitem{Simpson}
Levi Borevitz, Naima Nader, Theodore~J. Sandstrom, Amelia Shapiro, Austyn
  Simpson, and Jenna Zomback.
\newblock On localization of tight closure in line-{$S4$} quartics.
\newblock Available at https://arxiv.org/abs/2211.03220.

\bibitem{BrennerMonsky}
Holger Brenner and Paul Monsky.
\newblock Tight closure does not commute with localization.
\newblock {\em Ann. of Math. (2)}, 171(1):571--588, 2010.

\bibitem{Burch}
Lindsay Burch.
\newblock Codimension and analytic spread.
\newblock {\em Proc. Cambridge Philos. Soc.}, 72:369--373, 1972.

\bibitem{CaminataDe}
Alessio Caminata and Alessandro De~Stefani.
\newblock Frobenius, tight closure, and {F}-singularities.
\newblock Lecture notes available at
  https://www.dima.unige.it/{$\sim$}destefani/files/Notes.pdf.

\bibitem{Dade}
Everett~C. Dade.
\newblock {\em M{ULTIPLICITY} {AND} {MONOIDAL} {TRANSFORMATIONS}}.
\newblock ProQuest LLC, Ann Arbor, MI, 1960.
\newblock Thesis (Ph.D.)--Princeton University.

\bibitem{DHNB}
Alessandro De~Stefani, Craig Huneke, and Luis N\'{u}\~{n}ez Betancourt.
\newblock Frobenius {B}etti numbers and modules of finite projective dimension.
\newblock {\em J. Commut. Algebra}, 9(4):455--490, 2017.

\bibitem{SPYGlobal}
Alessandro De~Stefani, Thomas Polstra, and Yongwei Yao.
\newblock Global {F}robenius {B}etti numbers and {F}robenius {E}uler
  characteristics.
\newblock {\em Michigan Math. J.}, 71(3):533--552, 2022.

\bibitem{Dutta}
Sankar~P. Dutta.
\newblock Frobenius and multiplicities.
\newblock {\em J. Algebra}, 85(2):424--448, 1983.

\bibitem{EnescuShimomoto}
Florian Enescu and Kazuma Shimomoto.
\newblock On the upper semi-continuity of the {H}ilbert-{K}unz multiplicity.
\newblock {\em J. Algebra}, 285(1):222--237, 2005.

\bibitem{EquimultiplicityBook}
M.~Herrmann, S.~Ikeda, and U.~Orbanz.
\newblock {\em Equimultiplicity and blowing up}.
\newblock Springer-Verlag, Berlin, 1988.
\newblock An algebraic study, With an appendix by B. Moonen.

\bibitem{Hironaka}
Heisuke Hironaka.
\newblock Resolution of singularities of an algebraic variety over a field of
  characteristic zero. {I}, {II}.
\newblock {\em Ann. of Math. (2)}, 79:109--326, 1964.

\bibitem{Hochster}
Melvin Hochster.
\newblock Foundations of tight closure.
\newblock Lecture notes for Math 711, Fall 2007 available at
  http://www.math.lsa.umich.edu/{$\sim$}hochster/711F07/711.html.

\bibitem{HochsterHuneke1}
Melvin Hochster and Craig Huneke.
\newblock Tight closure, invariant theory, and the {B}rian\c con-{S}koda
  theorem.
\newblock {\em J. Amer. Math. Soc.}, 3(1):31--116, 1990.

\bibitem{HHSmooth}
Melvin Hochster and Craig Huneke.
\newblock {$F$}-regularity, test elements, and smooth base change.
\newblock {\em Trans. Amer. Math. Soc.}, 346(1):1--62, 1994.

\bibitem{HochsterYao}
Melvin Hochster and Yongwei Yao.
\newblock The {F}-rational signature and drops in the {H}ilbert-{K}unz
  multiplicity.
\newblock {\em Algebra Number Theory}, 16(8):1777--1809, 2022.

\bibitem{HunekeDseq}
Craig Huneke.
\newblock The theory of d-sequences and powers of ideals.
\newblock {\em Advances in Mathematics}, 46(3):249--279, 1982.

\bibitem{HunekeSurvey}
Craig Huneke.
\newblock Hilbert-{K}unz multiplicity and the {F}-signature.
\newblock In {\em Commutative algebra}, pages 485--525. Springer, New York,
  2013.

\bibitem{HunekeLeuschke}
Craig Huneke and Graham~J. Leuschke.
\newblock Two theorems about maximal {C}ohen-{M}acaulay modules.
\newblock {\em Math. Ann.}, 324(2):391--404, 2002.

\bibitem{HunekeSwanson}
Craig Huneke and Irena Swanson.
\newblock {\em Integral closure of ideals, rings, and modules}, volume 336 of
  {\em London Mathematical Society Lecture Note Series}.
\newblock Cambridge University Press, Cambridge, 2006.

\bibitem{HunekeYao}
Craig Huneke and Yongwei Yao.
\newblock Unmixed local rings with minimal {H}ilbert-{K}unz multiplicity are
  regular.
\newblock {\em Proc. Amer. Math. Soc.}, 130(3):661--665, 2002.

\bibitem{Kunz1}
Ernst Kunz.
\newblock Characterizations of regular local rings of characteristic {$p$}.
\newblock {\em Amer. J. Math.}, 91:772--784, 1969.

\bibitem{Kunz2}
Ernst Kunz.
\newblock On {N}oetherian rings of characteristic {$p$}.
\newblock {\em Amer. J. Math.}, 98(4):999--1013, 1976.

\bibitem{LL}
Tomas Larfeldt and Christer Lech.
\newblock Analytic ramifications and flat couples of local rings.
\newblock {\em Acta Math.}, 146(3-4):201--208, 1981.

\bibitem{Lipman}
Joseph Lipman.
\newblock Equimultiplicity, reduction, and blowing up.
\newblock In {\em Commutative algebra ({F}airfax, {V}a., 1979)}, volume~68 of
  {\em Lecture Notes in Pure and Appl. Math.}, pages 111--147. Dekker, New
  York, 1982.

\bibitem{Lyu}
Shiji Lyu.
\newblock Uniform bounds in excellent rings and applications to semicontinuity.
\newblock Preprint.

\bibitem{MaPolstraBook}
Linquan Ma and Thomas Polstra.
\newblock F-singularities: a commutative algebra approach.
\newblock Lecture notes available at
  https://www.math.purdue.edu/{$\sim$}ma326/F-singularitiesBook.pdf.

\bibitem{MPST}
Linquan Ma, Thomas Polstra, Karl Schwede, and Kevin Tucker.
\newblock {$F$}-signature under birational morphisms.
\newblock {\em Forum Math. Sigma}, 7:Paper No. e11, 20, 2019.

\bibitem{MQSLength2}
Linquan Ma, Pham~Hung Quy, and Ilya Smirnov.
\newblock Colength, multiplicity, and ideal closure operations {II}.
\newblock Available at https://arxiv.org/abs/2305.12469.

\bibitem{Monsky}
Paul Monsky.
\newblock The {H}ilbert-{K}unz function.
\newblock {\em Math. Ann.}, 263(1):43--49, 1983.

\bibitem{MonskyQL}
Paul Monsky.
\newblock Hilbert-{K}unz functions in a family: line-{$S_4$} quartics.
\newblock {\em J. Algebra}, 208(1):359--371, 1998.

\bibitem{MonskyQP}
Paul Monsky.
\newblock Hilbert-{K}unz functions in a family: point-{$S_4$} quartics.
\newblock {\em J. Algebra}, 208(1):343--358, 1998.

\bibitem{Polstra}
Thomas Polstra.
\newblock Uniform bounds in {F}-finite rings and lower semi-continuity of the
  {F}-signature.
\newblock {\em Trans. Amer. Math. Soc.}, 370(5):3147--3169, 2018.

\bibitem{PolstraSmirnov}
Thomas Polstra and Ilya Smirnov.
\newblock Equimultiplicity theory of strongly {$F$}-regular rings.
\newblock {\em Michigan Math. J.}, 70(4):837--856, 2021.

\bibitem{PolstraTucker}
Thomas Polstra and Kevin Tucker.
\newblock F-signature and {H}ilbert--{K}unz multiplicity: a combined approach
  and comparison.
\newblock {\em Algebra Number Theory}, 12(1):61--97, 2018.

\bibitem{Sannai}
Akiyoshi Sannai.
\newblock On dual {$F$}-signature.
\newblock {\em Int. Math. Res. Not. IMRN}, (1):197--211, 2015.

\bibitem{SmirnovSemi}
Ilya Smirnov.
\newblock Upper semi-continuity of the {Hilbert}-{Kunz} multiplicity.
\newblock {\em Compos. Math.}, 152(3):477--488, 2016.

\bibitem{SmirnovEqui}
Ilya Smirnov.
\newblock Equimultiplicity in {Hilbert}-{Kunz} theory.
\newblock {\em Math. Z}, 291(1-2):245--278, 2019.

\bibitem{Affine}
Ilya Smirnov.
\newblock On semicontinuity of multiplicities in families.
\newblock {\em Doc. Math.}, 25:381--399, 2020.

\bibitem{SmirnovTucker}
Ilya Smirnov and Kevin Tucker.
\newblock The theory of {F}-rational signature.
\newblock Preprint, available at https://arxiv.org/abs/1911.02642.

\bibitem{SmithVdB}
Karen~E. Smith and Michel Van~den Bergh.
\newblock Simplicity of rings of differential operators in prime
  characteristic.
\newblock {\em Proc. London Math. Soc. (3)}, 75(1):32--62, 1997.

\bibitem{Tucker}
Kevin Tucker.
\newblock {$F$}-signature exists.
\newblock {\em Invent. Math.}, 190(3):743--765, 2012.

\bibitem{WatanabeYoshida}
Kei-ichi Watanabe and Ken-ichi Yoshida.
\newblock Hilbert-{K}unz multiplicity and an inequality between multiplicity
  and colength.
\newblock {\em J. Algebra}, 230(1):295--317, 2000.

\bibitem{Yao}
Yongwei Yao.
\newblock Observations on the {$F$}-signature of local rings of characteristic
  {$p$}.
\newblock {\em J. Algebra}, 299(1):198--218, 2006.

\end{thebibliography}

\end{document}